\documentclass[10pt,a4paper,leqno,twoside]{amsart}
\usepackage[a4paper,left=3cm,top=4cm,bottom=5cm,right=3cm]{geometry}
\usepackage[T1]{fontenc}
\usepackage{amssymb,bbm}
\usepackage{amsmath,amsthm,dsfont}
\usepackage{amsfonts,mathrsfs}
\usepackage{latexsym}
\usepackage[cp1250]{inputenc}
\usepackage{graphicx}
\usepackage{indentfirst}
\usepackage{enumerate}
\usepackage{geometry}
\usepackage{enumitem}
\usepackage{xparse}
\usepackage{etoolbox}
\usepackage{mathptmx}
\usepackage{fancyhdr}
\usepackage[usenames,dvipsnames]{xcolor}
\usepackage[hidelinks]{hyperref} 

\bibliography{./IEEEabrv,./references/ref}

\patchcmd{\abstract}{\scshape\abstractname}{\textbf{\abstractname}}{}{}

\makeatletter
\DeclareMathAlphabet{\mathcal}{OMS}{cmsy}{m}{n}

\DeclareSymbolFont{operators}{OT1}{ztmcm}{m}{n}
\DeclareSymbolFont{letters}{OML}{ztmcm}{m}{it}
\DeclareSymbolFont{symbols}{OMS}{ztmcm}{m}{n}
\DeclareSymbolFont{largesymbols}{OMX}{ztmcm}{m}{n}
\DeclareSymbolFont{bold}{OT1}{ptm}{bx}{n}
\DeclareSymbolFont{italic}{OT1}{ptm}{m}{it}
\@ifundefined{mathbf}{}{\DeclareMathAlphabet{\mathbf}{OT1}{ptm}{bx}{n}}
\@ifundefined{mathit}{}{\DeclareMathAlphabet{\mathit}{OT1}{ptm}{m}{it}}
\DeclareMathSymbol{\omicron}{0}{operators}{`\o}

\DeclareMathAlphabet{\mathpzc}{OT1}{pzc}{m}{it}
\DeclareSymbolFont{operators}{OT1}{txr}{m}{n}
\SetSymbolFont{operators}{bold}{OT1}{txr}{bx}{n}
\def\operator@font{\mathgroup\symoperators}
\DeclareSymbolFont{italic}{OT1}{txr}{m}{it}
\SetSymbolFont{italic}{bold}{OT1}{txr}{bx}{it}
\DeclareSymbolFontAlphabet{\mathrm}{operators}
\DeclareMathAlphabet{\mathbf}{OT1}{txr}{bx}{n}
\DeclareMathAlphabet{\mathit}{OT1}{txr}{m}{it}
\SetMathAlphabet{\mathit}{bold}{OT1}{txr}{bx}{it}
\DeclareSymbolFont{letters}{OML}{txmi}{m}{it}
\SetSymbolFont{letters}{bold}{OML}{txmi}{bx}{it}
\DeclareFontSubstitution{OML}{txmi}{m}{it}
\DeclareSymbolFont{lettersA}{U}{txmia}{m}{it}
\SetSymbolFont{lettersA}{bold}{U}{txmia}{bx}{it}
\DeclareFontSubstitution{U}{txmia}{m}{it}
\DeclareSymbolFontAlphabet{\mathfrak}{lettersA}
\DeclareSymbolFont{symbols}{OMS}{txsy}{m}{n}
\SetSymbolFont{symbols}{bold}{OMS}{txsy}{bx}{n}
\DeclareFontSubstitution{OMS}{txsy}{m}{n}
\makeatother

\makeatletter
\renewcommand\abstractname{\scshape\bfseries Abstract}

\renewenvironment{proof}[1][\proofname]{\par \pushQED{\qed} \normalfont
  \topsep6\p@\@plus6\p@ \trivlist \itemindent\z@
  \item[\hskip\labelsep\bfseries
    #1\@addpunct{.}]\ignorespaces
}{
  \popQED\endtrivlist\@endpefalse
}
\makeatother

\makeatletter
    \@addtoreset{equation}{section} 
    \renewcommand{\theequation}{{\thesection}.\@arabic\c@equation} 

\makeatother

\makeatletter

\def\section{\@ifstar\unnumberedsection\numberedsection}
\def\numberedsection{\@ifnextchar[
  \numberedsectionwithtwoarguments\numberedsectionwithoneargument}
\def\unnumberedsection{\@ifnextchar[
  \unnumberedsectionwithtwoarguments\unnumberedsectionwithoneargument}
\def\numberedsectionwithoneargument#1{\numberedsectionwithtwoarguments[#1]{#1}}
\def\unnumberedsectionwithoneargument#1{\unnumberedsectionwithtwoarguments[#1]{#1}}
\def\numberedsectionwithtwoarguments[#1]#2{%
  \ifhmode\par\fi
  \removelastskip
  \vskip 4ex\goodbreak
  \refstepcounter{section}%
  \noindent
  \begingroup
  \leavevmode\centering\scshape\bfseries
  \thesection.
  #2
  \par
  \endgroup
  \vskip 1ex\nobreak
  \addcontentsline{toc}{section}{%
    \protect\numberline{\thesection}%
    #1}%
  }
\def\unnumberedsectionwithtwoarguments[#1]#2{%
  \ifhmode\par\fi
  \removelastskip
  \vskip 2ex\goodbreak
  \noindent
  \begingroup
  \leavevmode\centering\scshape\bfseries
  \leavevmode\centering\scshape\bfseries
  #2
  \par
  \endgroup
  \vskip 1ex\nobreak
  \addcontentsline{toc}{section}{%
    #1}%
}
\makeatother

\makeatletter
\def\@seccntformat#1{\csname mythe#1\endcsname}

\let\latex@subsection\subsection
\def\subsection{\@ifstar{\refstepcounter{subsection}\latex@subsection*}{\latex@subsection}}
\def\@makechapterhead#1{%
  \vspace*{40\p@}%
  {\parindent \z@ \raggedright \normalfont
    \interlinepenalty\@M
    \Huge \bfseries #1\par \nobreak
    \vskip 40\p@
  }}
\let\latex@l@chapter\l@chapter
\def\l@chapter#1#2{\begingroup\let\numberline\@gobble\latex@l@chapter{#1}{#2}\endgroup}
\makeatother

\theoremstyle{plain}
\newtheorem{Th}{Theorem}[section]
\newtheorem{Prop}[Th]{Proposition}
\newtheorem{Lem}[Th]{Lemma}
\newtheorem{Cor}[Th]{Corollary}
\theoremstyle{definition}
\newtheorem{Rem}[Th]{Remark}
\newtheorem{Ex}[Th]{Example}

\newtheorem{Def}[Th]{Definition}

\def\noo{\partial}

\def\bf{\textbf}
\def\it{\textit}
\def\te{\textnormal}

\def\leq{\leqslant}
\def\geq{\geqslant}
\def\R{{\mathbb R}}
\def\N{{\mathbb N}}

\def\B{{\textit{B}}}
\def\D{{\mathrm{dom}}\,}
\def\E{{\mathrm{epi}}\,}
\def\G{{\mathrm{gph}}\,}
\def\I{{\mathrm{int}}\,}

\def\kr{\overline}

\begin{document}
\vspace*{-1.15cm}
\title{Representation of convex Hamilton--Jacobi equations\\ in optimal control theory}


\author{\vspace*{-0.2cm}{Arkadiusz Misztela \textdagger}\vspace*{-0.2cm}}
\thanks{\textdagger\, Institute of Mathematics, University of Szczecin, Wielkopolska 15, 70-451 Szczecin, Poland; e-mail: arke@mat.umk.pl}


\begin{abstract}
In the paper we study the following problem: given a Hamilton-Jacobi equation where the Hamiltonian is convex with respect to the last variable, are there any optimal control problems representing it? In other words, we search for an appropriately regular  dynamics and a Lagrangian that represents the Hamiltonian with given properties. This problem was lately researched by Frankowska-Sedrakyan (2014) and Rampazzo (2005). We introduce a new method to construct a representation of a wide class of Hamiltonians, wider than it was achieved before. Actually, we get two types of representations: with compact and noncompact  control set, depending on regularity of the Hamiltonian. We conclude the paper by proving the stability of representations.
\\
\vspace{0mm}

\hspace{-1cm}
\noindent  \bf{Keywords.} Hamilton--Jacobi equations, optimal control theory, representation of Hamiltonians, \\
\hspace*{-0.55cm} parametrization of set-valued maps, convex analysis.

\vspace{2mm}\hspace{-1cm}
\noindent \bf{Mathematics Subject Classification.} 26E25, 49L25, 34A60, 46N10.
\end{abstract}

\maketitle

\pagestyle{myheadings}  \markboth{\small{ARKADIUSZ MISZTELA}
}{\small{REPRESENTATION OF HAMILTONIAN}}

%

\vspace{-0.8cm}



\section{Introduction}

\noindent Let  $A$-control set, $f$-dynamics, $l$-Lagrangian , $g$-terminal condition be given. For every $(t_0,x_0)\in[0,T]\times\R^n$ we consider the Bolza optimal control problem
\begin{equation}\label{bocp}
\begin{split}
\mathrm{minimize}&\;\;\;\Gamma((x,a)(\cdot)):= g(x(T))+\int_{t_0}^Tl(t,x(t),a(t))\,dt,\\
\mathrm{subject\;\, to}&\;\;\;\dot{x}(t)=f(t,x(t),a(t)),\;\; a(t)\in A\;\;\mathrm{a.e.}
\;\;t\in[t_0,T],\\
\mathrm{and}&\;\;\;x(t_0)=x_0.
\end{split}\tag{$\mathcal{P}_{t_0,x_0}$}
\end{equation}
While studying the problem \eqref{bocp}, one usually requires $f$ to be such that to every measurable control functions $a\!:\![t_0,\!T]\!\to\! A$ corresponds exactly one absolutely continuous solution $x\!:\![t_0,\!T]\!\to\!\R^n$  of the equation
\begin{equation}\label{rrdf}
\dot{x}(t)=f(t,x(t),a(t)),\;\; a(t)\in A\;\;\mathrm{a.e.}\;\;t\in[t_0,T]
\end{equation}
with the initial condition $x(t_0)=x_0$.
This is guaranteed, for instance, by the local Lipschitz continuity and the sublinear growth of $f$ with respect to $x$.  Let the  terminal condition $g$ be lower semicontinuous and the control set $A$ be compact in the space $\R^m$. We assume the Lagrangian $l$ is locally Lipschitz continuous with respect to $x$.  The above assumptions concerning the triple $(A,f,l)$ and  the terminal condition $g$ allow us to determine correctly the value function  $V$ for any $(t_0,x_0)\in[0,T]\times\R^n$ using the formula
\begin{equation}\label{fwfl}
V(t_0,x_0):= \inf\,\{\,\Gamma((x,a)(\cdot))\mid (x,a)(\cdot)\,\; \te{is a solution of}\,\;\eqref{rrdf}\,\; \te{with}\;\;x(t_0)=x_0 \,\}
\end{equation}

We say that the problems of optimal control $\{\eqref{bocp}\mid(t_0,x_0)\in[0,T]\times\R^n\}$ represent the following Hamilton-Jacobi  equation
\begin{equation}\label{rowhj}
\begin{array}{rll}
-V_{t}+ H(t,x,-V_{x})=0 &\!\!\textnormal{in}\!\! & (0,T)\times\R^n, \\[0mm]
V(T,x)=g(x) & \!\!\textnormal{in}\!\! & \;\R^n,
\end{array}
\end{equation}
provided the value function \eqref{fwfl} is the unique viscosity solution of the equation \eqref{rowhj}.  The assumption about the triple $(A,f,l)$  and  the terminal condition $g$ imposed above, allow us to prove that the value function \eqref{fwfl} is the unique viscosity solution of the equation \eqref{rowhj} with  the Hamiltonian given by the formula
\begin{equation}\label{hfl}
H(t,x,p):=  \sup_{a\in A}\,\{\,\langle\, p\,,f(t,x,a)\,\rangle\,-\,l(t,x,a)\,\}.
\end{equation}
The details concerning this well-known model may be found in the wide monograph of Bardi and Capuzzo-Dolcetta \cite{B-CD}. The classic results about uniqueness of viscosity solutions of the equation \eqref{rowhj} are originated from Crandall-Lions \cite{C-L,C-L-87}, Crandall-Evans-Lions \cite{C-E-L}, Ishii \cite{HI0},  Lions-Perthame \cite{L-P} Barron-Jensen \cite{B-J-0,B-J}, Frankowska~\cite{HF}, Frankowska-Plaskacz-Rzeżuchowski \cite{F-P-Rz} and Subbotin \cite{IS}.

Now we shall state the converse problem that we are going to study in this article. We assume that given is a Hamilton-Jacobi equation  \eqref{rowhj} with lower semicontinuous terminal condition $g$ and the  Hamiltonian $H$ satisfying the \it{existence and uniqueness conditions}, i.e.
\begin{itemize}
\item[(EU1)] $H(\cdot,x,p)$ is measurable for all $x,p\in\R^n$ and $H(t,\cdot,\cdot)$ is continuous for all $t\in[0,T]$, moreover $H(t,x,\,\cdot\,)$ is convex and  Lipschitz continuous with the Lipschitz constant $c(t)(1+|x|)$ for all $t\in[0,T]$, $x\in\R^n$.
\item[(EU2)] $H(t,\,\cdot\,,p)$ is Lipschitz continuous on closed ball $\B(0,R)$ with the Lipschitz constant
$k_R(t)(1+|p|)$ for all $t\in[0,T]\setminus\mathcal{N}_R$, $p\in\R^n$ and $R>0$, where $\mathcal{N}_R$ is the null set.
\end{itemize}
In this pattern, one faces a natural question: do there exist optimal control problems $\{\eqref{bocp}\mid(t_0,x_0)\in[0,T]\times\R^n\}$  that represent the equation \eqref{rowhj}? In other words, we ask if there exist such a triple $(A,f,l)$ that satisfies the equation \eqref{hfl} and the conditions stated above. We call such a triple $(A,f,l)$ a  \it{faithful representation} of the Hamiltonian $H$ in the optimal control theory.
The use of the name ''faithful representation'' is justified by the fact that there is infinitely many triples $(A,f,l)$, satisfying the equation \eqref{hfl}, including the ones with totally irregular functions  $f,l$. The triple $(A,f,l)$, not necessarily regular and satisfying the equality \eqref{hfl}  is called a  \it{representation} of the Hamiltonian $H$.

Faithful representations of the Hamiltonians satisfying the existence and uniqueness conditions were studied  first by Rampazzo \cite{FR} and next by Frankowska-Sedrakyan \cite{F-S}. Rampazzo \cite{FR} focused on the research of continuous representations with respect to  $t$. Frankowska-Sedrakyan \cite{F-S}, in turn, studied measurable representations with respect to  $t$. Our results include continuous as well as measurable case. Earlier,  Ishii \cite{HI} proposed a representation involving continuous functions $f,l$ and expressed the solution of a stationary Hamilton-Jacobi equation as the value function of an associated infinite horizon optimal control problem. The lack of local Lipschitz function $f$ with respect to the variable $x$ in Ishii \cite{HI} paper causes that, in general, not to every control $u(\cdot)$ there corresponds exactly one trajectory $x(\cdot)$. This fact  causes a lot of troubles in applications. Our results imply that both the dynamics $f$ and the Lagrangian $l$ are locally Lipschitz continuous in $x$.  Theorems on representation are, in general, used to research the regularity of solutions of the equations \eqref{rowhj}. Therefore we need to construct functions $f,l$ as regular as it is possible.
More reference and comments concerning theorems on representations, one may find in \cite{F-S,FR}.

Now, we introduce the construction of faithful representation proposed by Rampazzo in \cite{FR} and compare it to our construction of faithful representation. We start with construction of Rampazzo, also used by Frankowska-Sedrakyan in \cite{F-S}. We consider the  Legendre-Fenchel transform of $H$  with respect to the last variable, called a Lagrangian,
\begin{equation}\label{tran1}
 L(t,x,v):= \sup_{p\in\mathbb{R}^{n}}\,\{\,\langle v,p\rangle-H(t,x,p)\,\}.
\end{equation}
It is possible for $L$ to achieve value~ $+\infty$. The set $\D\varphi=\{x\mid\varphi(x)\not=\pm\infty\}$ is called the \emph{effective domain} of $\varphi$. The condition (EU1) imposed on  $H$ implies that the set-valued map $F(t,x):=\D L(t,x,\cdot)$ has nonempty, bounded, convex values.  It might happen that  values of $F$ are not closed (see Ex.\,\ref{ex-3}). The closed values of $F$ are guaranteed by boundedness of the function $L(t,x,\cdot)$ on the effective domain $\D L(t,x,\cdot)$, assumed additionally by Rampazzo \cite{FR} and Frankowska-Sedrakyan \cite{F-S}. The condition (EU2) imposed on $H$ implies $F$ is locally Lipschitz continuous in the sense of Hausdorff's distance. By parametrization theorem \cite[Thm. 9.7.2]{A-F}, there exists single-valued map $f$, with the control set $A:=\B(0,1)\subset\R^{n}$, satisfying the local Lipschitz conditions with respect to $x$ and  $f(t,x,A)=F(t,x)$. If $H(t,x,\cdot)$ is finite and convex for each $(t,x)$, then by Fenchel-Moreau theorem we can retrieve~$H$ from $L$ by performing the Legendre-Fenchel transform for a second time:
\begin{equation}\label{tr2}
 H(t,x,p)=\sup_{v\in\mathbb{R}^{n}}\,\{\,\langle p,v\rangle-L(t,x,v)\,\}.
\end{equation}
We notice that by \eqref{tr2}, of the definition of $F$ and equality $f(t,x,A)=F(t,x)$ we get
\begin{equation*}
H(t,x,p)= \sup_{a\in A}\,\{\,\langle\, p\,,f(t,x,a)\,\rangle-L(t,x,f(t,x,a))\,\}.
\end{equation*}
It means the triple $(A,f,l)$ is the representation of $H$ provided $l(t,x,a):= L(t,x,f(t,x,a))$. The question of function $l$ regularity still needs to be clarified. It turns out that with the above assumptions, the function $l$ may not be even continuous with respect to the variable $x$ (see Ex.\,\ref{prz-rep-010}).
It is the consequence of the fact that if $H$ satisfies the existence and uniqueness conditions, then the function $(x,v)\to L(t,x,v)$  is, in general, lower semicontinuous, even if the Lagrangian is bounded on the effective domain (see Ex.\,\ref{ex-1}). Therefore Rampazzo \cite{FR} and Frankowska-Sedrakyan \cite{F-S} introduced an additional assumption (H5). This assumption together with the assumption (EU2) implies that the function $(x,v)\to L(t,x,v)$ is Lipschitz continuous on $(\B(0,R)\times\R^n)\cap \D L(t,\cdot,\cdot)$ for all $t\in[0,T]\setminus\mathcal{N}_R$ and $R>0$. Hence, in an obvious way we obtain, that $l$ satisfies the local Lipschitz continuity with respect to $x$. Examples show that (H5) is a strong assumption. Hence Rampazzo states in \cite{FR} an open question -- how the assumption (H5) could be weakened. Another problem connected to the considered construction of the faithful representation concerns the unboundedness on the Lagrangian on the effective domain. In this case, it might happen that $F$ does not have closed values (see Ex.\,\ref{ex-4}). Therefore the parametrization theorem \cite[Thm. 9.7.2]{A-F}, utilized to construct the function $f$ may not be applied to $F$. Our construction of faithful representation, to which we now move, nicely handles indicated problems.

Let us define the set-valued map $E_L:[0,T]\times\R^n\multimap\R^n\times\R$ by the following formula
\begin{equation*}
E_L(t,x):= \{\,(v,\eta)\in\R^n\times\R\mid L(t,x,v)\leq \eta\,\}.
\end{equation*}
The set $\E \varphi=\{(x,\eta)\mid\varphi(x)\leq \eta\}$ is called the \emph{epigraph} of $\varphi$. We notice that $E_L(t,x)=\E L(t,x,\cdot)$. The condition (EU1) imposed on  $H$ implies $E_L$ has nonempty, closed, convex values. Off course, $E_L$ has unbounded values. The condition (EU2) imposed on $H$ implies $E_L$ is locally Lipschitz continuous in the sense of Hausdorff's distance. This fact is known, e.g. see \cite[Lem. 2 in Sec. 4.2]{FC}. From parametrization theorem \cite[Thm. 9.7.1]{A-F}, there follows the existence of a single-valued map $e:[0,T]\times\R^n\times A\rightarrow\R^n\times\R$, with the control set $A:=\R^{n+1}$ that satisfies the local Lipschitz continuity with respect to $x$ and
\begin{equation*}\label{rep_epi}
e(t,x,A)=E_L(t,x).
\end{equation*}
We define the function $f:[0,T]\times\R^n\times A\rightarrow\R^n$ and $l:[0,T]\times\R^n\times A\rightarrow\R$ by formulae:
\begin{equation}\label{def-fl}
f(t,x,a):=\pi_v(e(t,x,a))\;\;\;\te{and}\;\;\;l(t,x,a):=\pi_\eta(e(t,x,a)),
\end{equation}
where $\pi_v(v,\eta)=v$ and $\pi_\eta(v,\eta)=\eta$ for all $v\in\R^n$ and $\eta\in\R$. Then for any $t\in[0,T]$, $x\in\R^n$, $a\in A$ the following equality holds
\begin{equation*}
e(t,x,a)=(f(t,x,a),l(t,x,a)).
\end{equation*}
It follows from the above that the functions  $f,l$ have the same properties as the function $e$. Moreover, we prove that the triple $(A,f,l)$ constructed like above, is the representation of $H$ and  $f(t,x,A)=\D L(t,x,\cdot)$ (see Prop. \ref{prop-reprezentacja H}). Let us notice that for the above construction of the faithful representation, the assumption on Lagrangian boundedness on the effective domain as well as the assumption (H5) are superfluous. However, the control set is noncompact.
This result allows us to study the Hamilton-Jacobi equations~\eqref{rowhj} with Hamiltonians whose Lagrangians are not bounded on effective domain.
In addition to this, it gives the base for research of Hamilton-Jacobi equations~\eqref{rowhj} with Hamiltonians satisfying even more general conditions of existence and uniqueness studied in articles~\cite{DM-F-V,G,AM2,P-Q, FR0}.

The construction of the faithful representation with a compact control set is more complicated than the previous one. It turns out that for the Hamiltonians satisfying the existence and uniqueness conditions, the faithful representation with a compact control set  not always exist. The necessary condition for existence of the faithful representation with a compact control set  is the boundedness of the Lagrangian on the effective domain. To be more precise, we prove that if there exists a faithful representation $(A,f,l)$ of the Hamiltonian $H$ with a compact control set $A$, then there exists a function $\lambda:[0,T]\times\R^n\to\R$ with $\lambda(t,x)\geq L(t,x,v)$ for every $t\in[0,T]$, $x\in\R^n$, $v\in\D L(t,x,\cdot)$ and  $\lambda(t,\cdot)$ is locally Lipschitz continuous for all $t\in[0,T]$ (see Thm.\,\ref{podr22_th_wk}). It means that if the  Lagrangian is not bounded on the effective domain (see Ex.\,\ref{ex-4}), then there are no faithful representation with a compact control set. Hence, as it follows from our results, the assumption on boundedness of the Lagrangian on the effective domain proposed by  Rampazzo \cite{FR} and Frankowska-Sedrakyan \cite{F-S} was not superfluous.

Now we describe the construction of the faithful representation with a compact control set. We assume that the Lagrangian $L$ is bounded on the effective domain by the function $\lambda$.  We define the set-valued map $E_{\lambda,L}:[0,T]\times\R^n\multimap\R^n\times\R$ by the following formula
\begin{equation*}
E_{\lambda,L}(t,x):=\{\,(v,\eta)\in\R^n\times\R\mid L(t,x,v)\leq \eta\leq \lambda(t,x)\,\}.
\end{equation*}
If, additionally, $H$ satisfies the existence and uniqueness conditions, we may show that $E_{\lambda,L}$ has nonempty, compact and convex values. However, $E_{\lambda,L}$ does not, in general, satisfy the local Lipschitz continuity in the sense of the Hausdorff's distance. Thus, applying parametrization theorem \cite[Thm. 9.7.2]{A-F} to $E_{\lambda,L}$, we will not achieve appropriately regular parametrization. Therefore, we propose other solution of this problem. It bases on constructing, with the use of the technics appearing in the proofs of parametrization theorems \cite[Thm. 9.7.1 and 9.7.2]{A-F}, a function $e:[0,T]\times\R^n\times A\rightarrow\R^n\times\R$, with the control set $A:=\B(0,1)\subset\R^{n+1}$, satisfying the local Lipschitz continuity with respect to $x$ and
\begin{equation}\label{zaw-rpr-ep}
E_{\lambda,L}(t,x)\subset e(t,x,A)\subset E_{L}(t,x).
\end{equation}
We notice that the function $e$ is neither a parametrization of $E_{\lambda,L}$ nor a parametrization of $E_{L}(t,x)$. Despite this, the inclusions \eqref{zaw-rpr-ep} suffice for the triple $(A,f,l)$  to be the representation of $H$, where  $f$ and $l$ are given by \eqref{def-fl} (see Prop. \ref{prop-reprezentacja H-ogr}).  Let us notice that for the above construction of the faithful representation, the assumption (H5) in \cite{F-S,FR} is superfluous. It means that our construction of the faithful representation solves the problem of Rampazzo stated in~\cite{FR}.

The main results of this article are proved while more general assumption then~(EU2) is supposed. Actually, we assume that the Hamiltonian~$H$ satisfies a condition of the type
\begin{equation}\label{wanzm}
\begin{array}{l}
\forall\,R>0\;\exists\,k_R(\cdot)\;\exists\,w_R(\cdot,\cdot)\;\forall\,t\in[0,T]\setminus\mathcal{N}_R\;\forall\,x,y\in\B(0,R)\;\forall\,p\in\R^n\\
|\,H(t,x,p)-H(t,y,p)\,|\,\leq\, k_R(t)\cdot|p|\cdot|x-y|+w_R(t,|x-y|),
\end{array}\tag{EU3}
\end{equation}
where $w_R(t,\cdot\,)$ is local modulus. The above condition still guarantees the existence and uniqueness of viscosity solution of the equation~\eqref{rowhj} (see \cite{C-L-87,HI}). We notice that the Hamiltonian~$H$ given by the formula~\eqref{hfl} satisfies the condition~\eqref{wanzm}, if $f(t,\cdot,a)$ is $k_R(t)$-Lipschitz and $l(t,\cdot,a)$ is $w_R(t,\cdot\,)$-continuous on the ball $\B(0,R)$ for all $t\in[0,T]\setminus\mathcal{N}_R$, $a\in A$ and $R>0$. In fact, we are interested in a converse situation, i.e. if the Hamiltonian~$H$ satisfies conditions~(EU1) and~\eqref{wanzm}, does there exists a representation~$(A,f,l)$ of the Hamiltonian~$H$ with the functions~$f,l$ with the above properties. Well, our results contained in Theorems~\ref{th-rprez-glo} and~\ref{th-rprez-glo12} do not answer this question. Indeed, functions~$f$ appearing in these Theorems do not satisfy the local Lipschitz condition with respect to~$x$. Right now, we work on the solution of this problem that lies, as we suppose, in the regularity of Steiner selection. We know the condition~\eqref{wanzm} implies that the set-valued map~$x\to\D L(t,x,\cdot)$ is locally Lipschitz continuous in the sense of the Hausdorff's distance. This fact is a corollary of Theorem~\ref{tw2_rlhmh} (see Rem. \ref{remdlsh}). In addition to this, it follows from our construction of a representation that~$f(t,x,A)=\D L(t,x,\cdot)$. Therefore~$f$ parameterizes the effective domain of the Lagrangian. Unfortunately, this does not imply that~$f$ satisfies the local Lipschitz condition with respect to~$x$. We still suppose that the function~$f$ with such a property may be constructed using a similar recipe like the one contained in this paper.

Another result we have obtained is the stability of the faithful representation constructed above. In Section \ref{thms-stab} we show that if $H_i$ converge uniformly to $H$ on compacts in $[0,T]\times\R^n\times \R^n$, then $f_i$ converge to $f$ and $l_i$ converge to $l$ uniformly on compacts in $[0,T]\times\R^n\times A$. The proof of this fact bases on Wijsman's Theorem~\cite[Thm. 11.34]{R-W}.

The outline of the paper is as follows. Section \ref{section-2} contains hypotheses and preliminary results. In Section~\ref{section-3} we gathered our main results. Sections \ref{wk-kon-istr}, \ref{section-5}, \ref{thms-stab} contain the proofs of results from Section~3.


\pagebreak

\section{Hypotheses and preliminary results}\label{section-2}

\noindent An extended-real-valued function is called \it{proper} if it never takes  the value $-\infty$ and is not identically equal to $+\infty$. If  $H(t,x,\cdot)$ is proper, convex and lower semicontinuous for each $(t,x)$, then the Lagrangian $L(t,x,\cdot)$ given by \eqref{tran1} also has  these properties (see \cite[Thm. 11.1]{R-W} ). Furthermore, by \cite[Thm. 11.1]{R-W} we have  $H(t,x,\cdot\,)=L^{\ast}(t,x,\cdot\,)$ and $L(t,x,\cdot\,)=H^{\ast}(t,x,\cdot\,)$,  where $^{\ast}$ denotes the Legendre-Fenchel transform. We denote by
$\langle v,p\rangle$ the scalar product of $v,p\in\R^n$ and by $|p|$ the Euclidean norm of $p$.

\vspace{1mm}

\noindent\bf{\textsf{Let us describe the hypotheses needed in this paper.}}
\begin{enumerate}
\item[\te{(H1)}] $H:[0,T]\times\mathbb{R}^{n}\times\mathbb{R}^{n}\rightarrow\mathbb{R}$
is $t-$measurable for any $(x,p)\in\R^n\times\R^n$;
\item[\te{(H2)}] $H(t,x,p)$ is continuous with respect to $(x,p)$ for every $t\in[0,T]$;
\item[\te{(H3)}] $H(t,x,p)$ is convex with respect to $p$ for every $(t,x)\in[0,T]\times\R^n$;
\item[\te{(H4)}] There exists a measurable map $c:[0,T]\to\R_+$ such that for any $(t,x)\in[0,T]\times\R^n$\\ and for all $p,q\in\R^n$ it satisfies
\;$|H(t,x,p)-H(t,x,q)|\;\leq\; c(t)(1+|x|)|p-q|$.
\end{enumerate}

\begin{Prop}\label{prop2-fmw} We suppose that $H$ satisfies \te{(H1)$-$(H3)}. If $L$ is given by the formula \eqref{tran1}, then
\begin{enumerate}
\item[\te{(L1)}] $L:[0,T]\times\mathbb{R}^{n}\times\mathbb{R}^{n}\rightarrow\mathbb{R}\cup\{+\infty\}$ is Lebesgue$-$Borel$-$Borel measurable;
\item[\te{(L2)}] $L(t,x,v)$  is  lower semicontinuous  with respect to $(x,v)$ for every $t\in[0,T]$;
\item[\te{(L3)}] $L(t,x,v)$ is convex and proper with respect to $v$ for every $(t,x)\in[0,T]\times\R^n$;
\item[\te{(L4)}] $\forall\,(t,x,v)\in[0,T]\times\R^n\times\R^n\;\;\forall\,x_i\rightarrow x\;\;
\exists\,v_i\rightarrow v\;:\;L(t,x_i,v_i)\rightarrow L(t,x,v)$;
\item[]\hspace{-5mm}\textsf{Additionally, if $H$ satisfies \te{(H4)}, then}
\item[\te{(L5)}] $\forall\,(t,x,v)\in[0,T]\times\R^n\times\R^n \;:\; |v|>c(t)(1+|x|)\;\Rightarrow\; L(t,x,v) =+\infty$;
\item[]\hspace{-5mm}\textsf{Additionally, if $H$ is continuous, then $L$ is lower semicontinuous and}
\item[\te{(L6)}] $\forall\,(t,x,v)\in[0,T]\times\R^n\times\R^n\;\;\forall\,(t_i,x_i)\rightarrow (t,x)\;\;
\exists\,v_i\rightarrow v\;:\;L(t_i,x_i,v_i)\rightarrow L(t,x,v)$.
\end{enumerate}
\end{Prop}

Proposition \ref{prop2-fmw} can be proven using well-known properties of the Legendre-Fenchel transform that can be found in \cite{R-W}.

\begin{Def}
We say that a  set-valued map $F:[0,T]\multimap\R^m$ is \it{measurable}, if for every
open set $U\subset\R^m$ the inverse image  $F^{-1}(U):= \{\,t\in[0,T]\mid F(t)\cap U\not=\emptyset\,\}$
is Lebesgue-measurable set.
\end{Def}

From the condition (L1) we gather that a set-valued map $t\to E_L(t,x)$ is measurable for all $x\in\R^n$. The conditions (L2)$-$(L3) imply that
the set $E_L(t,x)$ is nonempty, closed and convex for all $t\in[0,T]$, $x\in\R^n$. The set $\G F:= \{\,(z,y)\mid y\in F(z)\,\}$ is called a \it{graph} of the set-valued map $F$.

\begin{Def}
We say that a set-valued map $F:\R^n\multimap\R^m$ is \it{lower semicontinuous} in  Kuratowski's sense, if for every open set $U\subset\R^m$ the inverse image $F^{-1}(U)$ is an open set. It is equivalent to the following condition
\begin{equation*}\label{roz1-lkura}
\forall\,(z,y)\in\G F\;\;\forall\;z_i\rightarrow z\;\;\exists\;y_i\rightarrow y\;:
\; y_i\in F(z_i)\;\;\te{for all large}\;\;i\in\N.
\end{equation*}
\end{Def}

The condition (L4) means that a set-valued map $x\to E_L(t,x)$ is lower semicontinuous for every $t\in[0,T]$. The condition (L2) implies that a  set-valued map $x\to E_L(t,x)$ is not only closed-valued,
but also it has the closed graph for every $t\in[0,T]$. For a nonempty subset $K$ of $\R^n$ we define $\|K\|:=\sup_{x\in K}|x|$. From the condition (L5) we have that $\|\D L(t,x,\cdot)\|\leq c(t)(1+|x|)$ for every $t\in[0,T]$, $x\in\R^n$. If  $L$ is lower semicontinuous
and satisfies (L6), then the set-valued map $E_L$ has a closed graph and is lower semicontinuous. Combining the above facts we obtain the following corollary.

\begin{Cor}\label{wrow-wm}
We suppose that $H$ satisfies \te{(H1)$-$(H3)}. If $L$ is given by \eqref{tran1}, then
\begin{enumerate}
\item[\te{(M1)}] $E_L(t,x)$ is a nonempty, closed, convex subset of $\;\R^{n+1}$ for all $(t,x)\in[0,T]\times\R^n$;
\item[\te{(M2)}] $x\to E_L(t,x)$ has a closed graph for every $t\in[0,T]$;
\item[\te{(M3)}] $x\to E_L(t,x)$ is lower semicontinuous for every  $t\in[0,T]$;
\item[\te{(M4)}] $t\to E_L(t,x)$ is measurable for every $x\in\R^n$;
\item[]\hspace{-5mm}\textsf{Additionally, if $H$ satisfies \te{(H4)}, then}
\item[\te{(M5)}] $\|\D L(t,x,\cdot)\|\leq c(t)(1+|x|)$ for every $(t,x)\in[0,T]\times\R^n$;
\item[]\hspace{-5mm}\textsf{Additionally, if  $H$ is continuous, then}
\item[\te{(M6)}] $(t,x)\to E_L(t,x)$ has a closed graph and is lower semicontinuous.
\end{enumerate}
\end{Cor}

\subsection{Lipschitz set-valued map $\pmb{x\to E_L(t,x)}$}
In this section we present Hausdorff continuity of a set-valued map in  Lagrangian and Hamiltonian terms. The notations $B_R$ and $B(0,R)$ stand for the closed ball in $\R^n$ centered at zero and with radius $R\geq 0$. Additionally, we set $B:= B_1$. The function $w:[0,T]\times[0,\infty)\to[0,\infty)$ will be called \it{modulus}, if $w(\cdot,r)$ is measurable for any $r\in[0,\infty)$ and $w(t,\cdot)$ is  nondecreasing, continuous and subadditive function satisfying $w(t,0)=0$  for all $t\in[0,T]$.

\vspace{1mm}
\noindent\bf{\textsf{The main hypothesis of this paper.}}
\vspace{-1mm}
\begin{equation}\label{hip_l1}
\begin{array}{l}
\te{For any}\; \,R>0\, \;\te{there exists}\,\; k_R:[0,T]\to\R_+-\te{measurable}\;\,\te{and}\;\, w_R(\cdot,\cdot)-\te{modulus}\;\,\te{and}\\\mathcal{N}_R-\te{null set such that}\;\;
|\,H(t,x,p)-H(t,y,p)\,|\,\leq\, k_R(t)\,|p|\,|x-y|+w_R(t,|x-y|)\;\; \te{for all}\\ t\in[0,T]\setminus\mathcal{N}_R \;\;\te{and for every}\;\; x,y\in\B_R,\; p\in\R^n.
\end{array}
\tag{HLC}
\end{equation}

\noindent The above condition is needed for the uniqueness of the solution of the equation  \!\eqref{rowhj}. \!It is described in \cite{AM0}.
\begin{Th}\label{tw2_rlhmh}
We suppose that a function $p\rightarrow H(t,x,p)$ is finite and convex for every
$(t,x)\!\in\! [0,\!T]\times\R^n$. Let  $L(t,x,\cdot\,)=H^{\ast}(t,x,\cdot\,)$ and  $H(t,x,\cdot\,)=L^{\ast}(t,x,\cdot\,)$ for every $(t,x)\in [0,T]\times\R^n$. Then we have
equivalences $\te{(HLC)}\Leftrightarrow\te{(LLC)}\Leftrightarrow\te{(MLC)}$\,\te{:}
\begin{equation}\label{hip_l2}
\begin{array}{l}
\it{For any}\; \,R>0\, \;\it{there exists}\,\; k_R:[0,T]\to\R_+-\it{measurable}\;\,\it{and}\;\; w_R(\cdot,\cdot)-\it{modulus}\;\, \it{and}\\ \mathcal{N}_R-\it{null set satisfying the condition}:\;\; \forall\,t\in[0,T]\setminus\mathcal{N}_R\;\, \forall\, x,y\in \B_R\;\,
\forall\, v\in\D L(t,x,\cdot)\\ \exists\,u\in\D L(t,y,\cdot)\;\; \it{such that}\;\; |u-v|\leq k_R(t)|y-x| \;\;\it{and}\;\; L(t,y,u)\leq L(t,x,v)+w_R(t,|x-y|).
\end{array}
\tag{LLC}
\end{equation}
\begin{equation}\label{hip_l3}
\begin{array}{l}
\it{For any}\; \,R>0\, \;\it{there exists}\,\; k_R:[0,T]\to\R_+-\it{measurable}\;\, \it{and}\;\, w_R(\cdot,\cdot)-\it{modulus}\;\, \it{and}\\\mathcal{N}_R-\it{null set such that}\;\;
E_L(t,x)\subset E_L(t,y)+(\,k_R(t)\,|x-y|\,\B\,)\times(\,w_R(t,|x-y|)\,[-1,1]\,)\\ \it{for all}\;\, t\in[0,T]\setminus\mathcal{N}_R\;\, \it{and}\;\,\it{for every}\;\, x,y\in\B_R.
\end{array}
\tag{MLC}
\end{equation}
Equivalences hold for the same functions $k_R(\cdot)$, $w_R(\cdot,\cdot)$
and the set $\mathcal{N}_R$.
\end{Th}

Theorem \ref{tw2_rlhmh} is an extended version of Clarke's result (see \cite[Lem. 2 in Sec. 4.2]{FC}). Our Theorem \ref{tw2_rlhmh} follows from Propositions \ref{0rwwl1} and \ref{0lnrwwl1} that are proven below. Let $K$ be a nonempty subset of $\R^m$. The distance from $x\in\R^m$ to $K$ is defined by $d(x,K):=\inf_{y\in K}|x-y|$. For nonempty subsets $K$, $D$ of $\R^m$, the extended Hausdorff distance between $K$ and $D$ is defined by
\begin{equation}
\mathscr{H}(K,D):= \max\left\{\sup_{x\in K}d(x,D),\;\sup_{x\in D}d(x,K)\right\}\in\R\cup\{+\infty\}.
\end{equation}
By Theorem \ref{tw2_rlhmh} and the condition \eqref{hip_l3} we obtain the following corollary.

\begin{Cor}\label{hlc-cor-ner}
We suppose that  $H$ satisfies  \te{(H1)$-$(H3)} and \eqref{hip_l1}. If $L$ is given by \eqref{tran1}, then
for any $R>0$, for all $t\in[0,T]\setminus\mathcal{N}_R$ and for every $x,y\in \B_R$ we have
\begin{equation}\label{n2_hd}
\mathscr{H}(E_L(t,x),E_L(t,y))\leq 2k_R(t)\,|x-y|+2w_R(t,|x-y|).
\end{equation}
\end{Cor}

\begin{Rem}\label{remdlsh}
Using the condition \eqref{hip_l2} we easily infer that a set-valued map $F_L(t,x)=\D L(t,x,\cdot)$ satisfies the inequality  $\mathscr{H}(F_L(t,x),F_L(t,y))\leq k_R(t)\,|x-y|$ for all $t\in[0,T]\setminus\mathcal{N}_R$,  $x,y\in \B_R$ and $R>0$. This means that a set-valued map $x\to F_L(t,x)$ is $k_R(t)-$Lipschitz on the ball $\B_R$ for all  $t\in[0,T]\setminus\mathcal{N}_R$ and  $R>0$. We notice that passing from the inclusion  \eqref{hip_l3} to the inequality \eqref{n2_hd} we lose this property. Therefore, the inequality \eqref{hip_l1} is stronger than the inequality \eqref{n2_hd}.
\end{Rem}

The epi-sum of functions $\varphi_1,\varphi_2:\R^n\rightarrow\mathbb{R}\cup\{+\infty\}$ is a
function $\varphi_1\,\sharp\,\varphi_2:\R^n\rightarrow\mathbb{R}\cup\{\pm\infty\}$
given by the formula
\begin{equation*}
\varphi_1\,\sharp\,\varphi_2(v):= \inf_{u\in\R^n}\{\varphi_1(u)+\varphi_2(v-u)\}.
\end{equation*}

\begin{Lem}[\te{\cite[Thm. 11.23]{R-W}}]\label{0epi-suma-lem}
We suppose that functions $h_1,h_2:\R^n\rightarrow\R\cup\{+\infty\}$ are proper,
convex and lower semicontinuous. Let the set $\D h_2^{\ast}$ be bounded.
Then the epi-sum $h_1^{\ast}\,\sharp\, h_2^{\ast}$ is a proper, convex and lower semicontinuous function. Besides, we have the equality
\begin{equation}\label{0epi-suma-row}
(h_1+h_2)^{\ast}=h_1^{\ast}\,\sharp\, h_2^{\ast}.
\end{equation}
\end{Lem}

\begin{Prop}\label{0rwwl1}
We suppose that  $p\rightarrow H(t,x,p)$ and $p\rightarrow H(t,y,p)$ are proper, convex and lower semicontinuous. Let  $L(t,x,\cdot\,):=H^{\ast}(t,x,\cdot\,)$ and  $L(t,y,\cdot\,):=H^{\ast}(t,y,\cdot\,)$. Then the following conditions are equivalent:
\begin{enumerate}[leftmargin=2.7cm]
\item[$\pmb{\te{(a)}}$] $H(t,x,p)\leq H(t,y,p)+k_R(t)\,|p|\,|x-y|+w_R(t,|x-y|)$ for all $p\in\R^n$.
\vspace{1mm}
\item[$\pmb{\te{(b)}}$] For all\, $v\in\D L(t,x,\cdot)$\, there exists\, $u\in\D L(t,y,\cdot)$\, such that\\  $|u-v|\leq k_R(t)\,|x-y|$\, and\,  $L(t,y,u)\leq L(t,x,v)+w_R(t,|x-y|)$.
\end{enumerate}
\end{Prop}

\begin{proof}
We start with the implication  $\te{(a)}\Rightarrow\te{(b)}$. Let $h_1(p):=H(t,y,p)$, $h_2(p):=k_R(t)\,|p|\,|x-y|+w_R(t,|x-y|)$, $h_3(p):=h_1(p)+h_2(p)$, $h_4(p):=H(t,x,p)$ for
every $p\in\R^n$.
Of course, $h_1^{\ast}(v)=L(t,y,v)$ and $h_4^{\ast}(v)=L(t,x,v)$ for every $v\in\R^n$. It is not difficult to calculate the following for every $v\in\R^n$.
\begin{equation}\label{0prop-wa2}
h_2^{\ast}(v)=\left\{
\begin{array}{lll}
\:-w_R(t,|x-y|) & \te{if} & |v|\leq k_R(t)\,|x-y| \\
\:+\infty & \te{if} & |v|> k_R(t)\,|x-y|.
\end{array}
\right.
\end{equation}
we notice that functions $h_1,h_2$ satisfy assumptions of Lemma \ref{0epi-suma-lem}. Therefore, by the equality (\ref{0epi-suma-row}) we have for every  $v\in\R^n$,
\begin{equation}\label{0prop-wa1}
h_3^{\ast}(v)=(h_1+h_2)^{\ast}(v)=h_1^{\ast}\,\sharp\, h_2^{\ast}(v).
\end{equation}
By (\ref{0prop-wa2}), (\ref{0prop-wa1}) and the definition of the epi-sum we get
for every $v\in\R^n$
\begin{equation}\label{0prop-wa3}
h_3^{\ast}(v)=\inf_{u\,:\,|v-u|\leq k_R(t)|x-y|}\{\,L(t,y,u)-w_R(t,|x-y|)\,\}.
\end{equation}
The inequality $\te{(a)}$ implies $h_4(p)\leq h_3(p)$ for every $p\in\R^n$. Therefore,
by the property of the Legendre-Fenchel transform we obtain  $h_3^{\ast}(v)\leq h_4^{\ast}(v)$ for all $v\in\R^n$. Using the property (\ref{0prop-wa3}) we have for all $v\in\R^n$
\begin{equation}\label{0prop-wa4}
L(t,x,v)\geq \inf_{u\,:\,|v-u|\leq k_R(t)|x-y|}\{\,L(t,y,u)-w_R(t,|x-y|)\,\}.
\end{equation}
The function  $u\rightarrow L(t,y,u)-w_R(t,|x-y|)$ is lower semicontinuous, so it achieves its minimum
on the compact set $\{\,u\,\mid\,|v-u|\leq k_R(t)|x-y|\,\}$. Using the inequality  (\ref{0prop-wa4}), we obtain the condition $\te{(b)}$ from the proposition.

Now, we prove the implication $\te{(b)}\Rightarrow\te{(a)}$. To this end, we fix $\kr{p}\in\R^n$. If $H(t,y,\kr{p})=+\infty$, then the inequality $\te{(a)}$ holds. We suppose that  $H(t,y,\kr{p})<+\infty$. Then the value $H(t,y,\kr{p})$ is a real number, because the function $p\rightarrow H(t,y,p)$ is proper.  We suppose that $H(t,x,\kr{p})=+\infty$. Next, we set $M:=H(t,y,\kr{p})+k_R(t)\,|\kr{p}|\,|x-y|+w_R(t,|x-y|)$. Because $H(t,x,\cdot\,)=L^{\ast}(t,x,\cdot\,)$, there exists $\kr{v}\in\D L(t,x,\cdot)$
such that
\begin{equation}\label{0prop-wa5}
M<\langle \kr{p},\kr{v}\rangle-L(t,x,\kr{v}).
\end{equation}
By the condition $\te{(b)}$, there exists $\kr{u}\in\D L(t,y,\cdot)$ such that
\begin{equation}\label{0prop-wa6}
    |\kr{u}-\kr{v}|\leq k_R(t)\,|y-x| \;\;\te{and}\;\; L(t,y,\kr{u})\leq L(t,x,\kr{v})+w_R(t,|y-x|).
\end{equation}
By the inequality (\ref{0prop-wa5}) and (\ref{0prop-wa6}) we obtain
\begin{eqnarray*}
M &< & \langle \kr{p},\kr{v}\rangle-L(t,x,\kr{v})+H(t,y,\kr{p})-
\sup_{v\in\mathbb{R}^{n}}\,\{\langle \kr{p},v\rangle-L(t,y,v)\}\\
&\leq &  \langle \kr{p},\kr{v}\rangle-L(t,x,\kr{v})+ H(t,y,\kr{p})-\langle \kr{p},\kr{u}\rangle+L(t,y,\kr{u})\\
&\leq & H(t,y,\kr{p})+ |\kr{p}|\,|\kr{v}-\kr{u}|+L(t,y,\kr{u})-L(t,x,\kr{v})\\
&\leq & H(t,y,\kr{p})+k_R(t)\,|\kr{p}|\,|x-y|+w_R(t,|x-y|)\;\;=\;\;M.
\end{eqnarray*}
The above contradiction  means that $H(t,x,\kr{p})<+\infty$. Then the value $H(t,x,\kr{p})$ is a real number, because the function  $p\rightarrow H(t,x,p)$ is proper. We set $\varepsilon>0$. As $H(t,x,\cdot\,)=L^{\ast}(t,x,\cdot\,)$,  there exists $\kr{v}\in\D L(t,x,\cdot)$ such that
\begin{equation}\label{00prop-wa5}
    H(t,x,\kr{p})-\varepsilon\leq\langle \kr{p},\kr{v}\rangle-L(t,x,\kr{v}).
\end{equation}
By the condition $\te{(b)}$ there exists $\kr{u}\in\D L(t,y,\cdot)$ such that
\begin{equation}\label{00prop-wa6}
    |\kr{u}-\kr{v}|\leq k_R(t)\,|y-x| \;\;\te{and}\;\; L(t,y,\kr{u})\leq L(t,x,\kr{v})+w_R(t,|y-x|).
\end{equation}
By the inequality (\ref{00prop-wa5}) and (\ref{00prop-wa6}) we obtain
\begin{eqnarray*}
    H(t,x,\kr{p})-\varepsilon &\leq & \langle \kr{p},\kr{v}\rangle-L(t,x,\kr{v})+H(t,y,\kr{p})-
    \sup_{v\in\mathbb{R}^{n}}\,\{\langle \kr{p},v\rangle-L(t,y,v)\}\\
    &\leq &  \langle \kr{p},\kr{v}\rangle-L(t,x,\kr{v})+ H(t,y,\kr{p})-\langle \kr{p},\kr{u}\rangle+L(t,y,\kr{u})\\
    &\leq & H(t,y,\kr{p})+ |\kr{p}|\,|\kr{v}-\kr{u}|+L(t,y,\kr{u})-L(t,x,\kr{v})\\
    &\leq & H(t,y,\kr{p})+k_R(t)\,|\kr{p}|\,|x-y|+w_R(t,|x-y|).
\end{eqnarray*}
As $\varepsilon>0$ is an arbitrary number, we get $H(t,x,\kr{p}) \leq H(t,y,\kr{p})+k_R(t)\,|\kr{p}|\,|x-y|+w_R(t,|x-y|)$. Also, $\kr{p}\in\R^n$ is arbitrary, so we have the inequality $H(t,x,p) \leq H(t,y,p)+k_R(t)\,|p|\,|x-y|+w_R(t,|x-y|)$ for every $p\in\R^n$. It~ends the proof.
\end{proof}

\begin{Prop}\label{0lnrwwl1}
We suppose that  $v\rightarrow L(t,x,v)$ and $v\rightarrow L(t,y,v)$ are proper.  Then the following conditions are equivalent:
\begin{enumerate}[leftmargin=2.7cm]
\item[$\pmb{\te{(a)}}$] For all\, $v\in\D L(t,x,\cdot)$\, there exists\, $u\in\D L(t,y,\cdot)$\, such that\\  $|u-v|\leq k_R(t)\,|y-x|$\, and\,  $L(t,y,u)\leq L(t,x,v)+w_R(t,|y-x|)$.
\vspace{1mm}
\item[$\pmb{\te{(b)}}$] $E_L(t,x)\subset E_L(t,y)+(\,k_R(t)\,|x-y|\,\B\,)\times(\,w_R(t,|x-y|)\,[-1,1]\,)$.
\end{enumerate}
\end{Prop}

\begin{proof}
We start with the implication $\te{(a)}\Rightarrow\te{(b)}$. Without loss of generality
we assume that $x\not=y$. Let $(v,\eta)\in E_L(t,x)$. Then $L(t,x,v)\leq \eta$. Thus, $v\in\D L(t,x,\cdot)$, because the function $v\!\rightarrow\! L(t,x,v)$ is proper.
By the condition $\te{(a)}$, there exists $u\in\D L(t,y,\cdot)$ such that
\begin{equation*}
\te{(i)}\;\;|u-v|\leq k_R(t)\,|y-x| \;\;\te{and}\;\;\te{(ii)}\;\; L(t,y,u)\leq L(t,x,v)+w_R(t,|y-x|).
\end{equation*}
Let us define  $b\in\R^n$, $\mu\in\R$ i $s\in[-1,1]$ in the following way:
\begin{equation*}
b:=\frac{v-u}{k_R(t)\,|y-x|}\;\;\te{if}\;\;k_R(t)>0,\hspace{0.5cm}b:=0\;\;\te{if}\;\;k_R(t)=0,\hspace{0.5cm}\mu:=\eta+w_R(t,|y-x|),\hspace{0.5cm}s:=-1.
\end{equation*}
We notice that from (i) we have $b\in\B$. Besides, from (ii) we obtain $(u,\mu)\in\E L(t,y,\cdot)$,
because
\begin{equation*}
L(t,y,u)\leq L(t,x,v)+w_R(t,|y-x|)\leq\eta+w_R(t,|y-x|)=\mu.
\end{equation*}
Therefore, $(b,s)\in\B\times[-1,1]$ and $(u,\mu)\in E_L(t,y)$. Thus, we get
\begin{eqnarray*}
(v,\eta) &=& (u,\mu)+(k_R(t)\,|y-x|\,b\,,\,w_R(t,|y-x|)\,s)\\
&\in & E_L(t,y)+(\,k_R(t)\,|x-y|\,\B\,)\times(\,w_R(t,|x-y|)\,[-1,1]\,).
\end{eqnarray*}
Thus, the condition  $\te{(b)}$ of the proposition is proven.

Now, we work with the implication $\te{(b)}\Rightarrow\te{(a)}$. Let $v\in\D L(t,x,\cdot)$. Then
$(v,L(t,x,v))\in E_L(t,x)$. Therefore, by the condition  $\te{(b)}$ we obtain
$$(v,L(t,x,v))\in  E_L(t,y)+(\,k_R(t)\,|x-y|\,\B\,)\times(\,w_R(t,|x-y|)\,[-1,1]\,).$$
So, there exists $(u,\mu)\in E_L(t,y)$ and $(b,s)\in\B\times[-1,1]$ such that
\begin{eqnarray}\label{00arfr-1}
(v,L(t,x,v))=(u,\mu)+(k_R(t)\,|y-x|\,b\,,\,w_R(t,|y-x|)\,s).
\end{eqnarray}
$(u,\mu)\in E_L(t,y)$, so $L(t,y,u)\leq \mu$. Hence  $u\in\D L(t,y,\cdot)$, because
the function $v\rightarrow L(t,y,v)$ is proper. By the equality (\ref{00arfr-1}) we have  $|u-v|= k_R(t)\,|y-x|\,|b|\leq k_R(t)\,|y-x|$ and
\begin{eqnarray*}
L(t,y,u)\;\;\leq\;\;\mu &=& L(t,x,v)+w_R(t,|y-x|)(-s)\\
&\leq & L(t,x,v)+w_R(t,|y-x|).
\end{eqnarray*}
 Thus, we have proven that for every $v\in\D L(t,x,\cdot)$ there exists  $u\in\E L(t,y,\cdot)$ such that $|u-v|\leq k_R(t)\,|y-x|$ and  $L(t,y,u)\leq L(t,x,v)+w_R(t,|y-x|)$.
It ends the proof of the proposition.
\end{proof}

\subsection{Examples of Hamiltonians satisfying {\te{\pmb{(}\bf{H1}\pmb{)}$\pmb{-}$\pmb{(}\bf{H4}\pmb{)}}} and \pmb{(}\bf{HLC}\pmb{)}}\label{przy-podroz} \label{przyklady-sp}
In this subsection we present a few examples where the assumptions of our theorems about representation we shall state in the next section are satisfied. These examples have nonregular Lagrangians, so they do not fulfill conditions of theorems contained in  \cite{F-S,FR}. Basically, they show how general our results are comparing to those in
 \cite{F-S,FR}.

\begin{Ex}\label{ex-1}
Let us define the Hamiltonian $H:\R\times\R\rightarrow\R$ by the formula
\begin{equation*}
H(x,p):=\max\{\,|p|\,|x|-1,0\,\}.
\end{equation*}
This Hamiltonian satisfies conditions  (H1)$-$(H4) and \eqref{hip_l1}. Conditions (H4)
and \eqref{hip_l1} follow easily from the equality  $H(x,p)=(|px|-1+|1-|px||)/2$, that
is satisfied for all $x,p\in\R$. The Lagrangian  $L:\R\times\R\rightarrow\R\cup\{+\infty\}$
given by the formula (\ref{tran1}) has the form
\begin{equation*}
L(x,v)=\left\{
\begin{array}{ccl}
+\infty & \te{if} & v\not\in[-|x|,|x|\,],\;x\not=0,\\[1mm]
\left|\frac{\displaystyle v}{\displaystyle x}\right| & \te{if} & v\in[-|x|,|x|\,],\;x\not=0, \\[1mm]
0 & \te{if} & v=0,\; x=0,\\[0mm]
+\infty & \te{if} & v\not=0,\;x=0.
\end{array}
\right.
\end{equation*}
Obviously,  $\D L(x,\cdot)=[-|x|,|x|\,]$ for all $x\in\R$. Besides, the function  $(x,v)\rightarrow L(x,v)$ does not satisfy the assumption (H5) of~\cite{F-S,FR}. Indeed, it is not continuous on the set $\D L$, because
\begin{equation*}
\lim_{i\rightarrow\infty}L\left(1/i,1/i\right)=1\not= 0=L(0,0).
\end{equation*}
\end{Ex}

\begin{Ex}[Rampazzo]\label{ex-2}
Let us define the Hamiltonian $H:\R\times\R\rightarrow\R$ by the formula
\begin{equation*}
H(x,p):=\sqrt{1+p^2}-|x|.
\end{equation*}
This Hamiltonian satisfies assumptions (H1)$-$(H4) and \eqref{hip_l1}. The Lagrangian  $L:\R\times\R\rightarrow\R\cup\{+\infty\}$ given by the formula (\ref{tran1})
has the following form
\begin{equation*}
L(x,v)=\left\{
\begin{array}{ccl}
-\sqrt{1-v^2}+|x| & \te{if} & v\in[-1,1],\\[1mm]
+\infty & \te{if} & v\not\in[-1,1].
\end{array}
\right.
\end{equation*}
Obviously, $\D L(x,\cdot)=[-1,1]$ for all $x\in\R$. We notice that the function $(x,v)\rightarrow L(x,v)$ is continuous on the set $\D L$, but it does not fulfill the condition (H5) that can be found in \cite{F-S,FR}.
\end{Ex}

\begin{Ex}\label{ex-3}
Let us define the Hamiltonian $H:\R\times\R\rightarrow\R$ by the formula
\begin{equation*}
H(x,p):=\left\{
\begin{array}{lcl}
p-1-|x| & \te{if} & p\geq -1, \\[1mm]
-2\sqrt{-p}-|x| & \te{if} & p<-1.
\end{array}
\right.
\end{equation*}
This Hamiltonian satisfies assumptions (H1)$-$(H4) and \eqref{hip_l1}. The Lagrangian  $L:\R\times\R\rightarrow\R\cup\{+\infty\}$ given by the formula (\ref{tran1})
has the following form
\begin{equation*}
L(x,v)=\left\{
\begin{array}{lcl}
\frac{\displaystyle 1}{\displaystyle v}+|x| & \te{if} & v\in(0,1], \\[1mm]
+\infty & \te{if} & v\not\in(0,1].
\end{array}
\right.
\end{equation*}
The set  $\D L(x,\cdot)=(0,1]$ is neither closed nor open and the function $v\rightarrow L(x,v)$ on this set is not bounded for every  $x\in\R$.
It means that the Lagrangian does not satisfy conditions of~\cite{F-S,FR}.
\end{Ex}

\begin{Ex}\label{ex-4}
Let us define the Hamiltonian $H:\R\times\R\rightarrow\R$ by the formula
\begin{equation*}
H(x,p):=\left\{
\begin{array}{ccl}
(\sqrt{|xp|}-1)^2 & \te{if} & |xp|> 1, \\[1mm]
0 & \te{if} & |xp|\leq 1.
\end{array}
\right.
\end{equation*}
This Hamiltonian satisfies assumptions (H1)$-$(H4) and \eqref{hip_l1}. The Lagrangian  $L:\R\times\R\rightarrow\R\cup\{+\infty\}$ given by the formula (\ref{tran1})
has the following form
\begin{equation*}
L(x,v)=\left\{
\begin{array}{ccl}
+\infty & \te{if} & v\not\in(-|x|,|x|\,),\;x\not=0,\\[1mm]
\frac{\displaystyle |v|}{\displaystyle |x|-|v|} & \te{if} & v\in(-|x|,|x|\,),\;x\not=0, \\[2mm]
0 & \te{if} & v=0,\; x=0,\\[0mm]
+\infty & \te{if} & v\not=0,\;x=0.
\end{array}
\right.
\end{equation*}
The set $\D L(x,\cdot)=(-|x|,|x|\,)$  is not closed and the function $v\rightarrow L(x,v)$ is not bounded on this set for every $x\in\R\setminus\{0\}$. Besides, the function  $(x,v)\rightarrow L(x,v)$  is not continuous on the set $\D L$. It means that the Lagrangian does not satisfy
conditions of~\cite{F-S,FR}.
\end{Ex}

\begin{Ex}\label{ex-5}
Let us define the Hamiltonian $H:[0,T]\times\R\times\R\rightarrow\R$ by the formula
\begin{equation*}
H(t,x,p):=\frac{p^2}{2+2t}-|x|.
\end{equation*}
This Hamiltonian satisfies assumptions (H1)$-$(H3) and \eqref{hip_l1}. We notice that
it does not satisfy the condition (H4). Nevertheless, it fulfills weaker (than in the current paper) conditions of existence and uniqueness that were considered in
\cite{G,AM2,P-Q}. The Lagrangian  $L:[0,T]\times\R\times\R\rightarrow\R\cup\{+\infty\}$ given by the formula (\ref{tran1}) has the following form
\begin{equation*}
L(t,x,v)=(1+t)v^2+|x|.
\end{equation*}
Obviously, $\D L(t,x,\cdot)=\R$ for every $t\in[0,T]$, $x\in\R$. Besides, the function
$v\rightarrow L(t,x,v)$ is not bounded on $\D L(t,x,\cdot)$ for every $t\in[0,T]$, $x\in\R$. It means that Lagrangian does not satisfy
conditions of~\cite{F-S,FR}. Additionally, a set-valued map $t\to E_L(t,x)$
is not continuous in the sense of Hausdorff's distance.
\end{Ex}



\section{Main Results}\label{section-3}
\noindent
In this section we describe main results of the paper that concern  faithful representations of Hamiltonians satisfying the existence and uniqueness conditions. We start with proving that representations are not determined uniquely. In addition to this, they can be totally irregular.

We consider the  Hamiltonian $H:\R\times\R\rightarrow\R$ given by the formula  $H(x,p):=|p|$. We notice that the triple $([-1,1],f,l)$ is a representation of this  Hamiltonian if functions $f,l:\R\times [-1,1]\rightarrow\R$ satisfy the following conditions:
\begin{equation}\label{wardlap}
|f(x,a)|\leq 1,\; f(x,1)=1,\;f(x,-1)=-1\;\;\;\te{and}\;\;\; l(x,a)\geq 0,\; l(x,1)=l(x,-1)=0.
\end{equation}
Let $h(\cdot)$ and $k(\cdot)$ be arbitrary functions on $\R$ with values in  $[0,\infty)$.
Then
\begin{equation}\label{pfcn}
f_h(x,a):=a\,(1+|a|\,h(x))/(1+h(x)),\quad l_k(x,a):=(1-|a|)\,k(x),\quad x\in\R,\,a\in[-1,1]
\end{equation}
satisfy conditions \eqref{wardlap}. Therefore, every triple $([-1,1],f_h,l_k)$, where $f_h,l_k$ are given by \eqref{pfcn}, is a representation of the Hamiltonian $H(x,p)=|p|$. There exist also representations with nonmeasurable (with respect to the state variable) functions $f_h,l_k$, for instance if $h(\cdot)$ and $k(\cdot)$ are not measurable.
However, our results show that from  the set of representations one can always choose a faithful representation.

\subsection{Representation $\pmb{(A,f,l)}$  with noncompact set $\pmb{A}$} This subsection is devoted to the new representation theorem for convex Hamiltonians, with noncompact control sets.

\begin{Th}[\bf{Representation}]\label{th-rprez-glo}
We suppose that  $H$ satisfies  \te{(H1)$-$(H3)} and \eqref{hip_l1}. Then there exists  $f:[0,T]\times\R^n\times A\rightarrow\R^n$ and $l:[0,T]\times\R^n\times A\rightarrow\R$, measurable in $t$ for all $(x,a)\in\R^n\times A$ and continuous in $(x,a)$ for all $t\in[0,T]$, with the control set $A:=\R^{n+1}$, such that for every $t\in[0,T]$, $x,p\in\R^n$
\begin{equation*}
 H(t,x,p)=\sup_{\;\;a\in \R^{n+1}}\!\!\{\,\langle\, p,f(t,x,a)\,\rangle-l(t,x,a)\,\}
\end{equation*}
and $f(t,x,A)=\D H^{\ast}(t,x,\cdot)$. Moreover\, $l(t,x,a)\geq -|H(t,x,0)|$\, for all\,  $t\in[0,T]$, $x\in \R^n$, $a\in R^{n+1}$.

\noindent Furthermore, for all $R>0$, $t\in[0,T]\setminus\mathcal{N}_R$,\, $x\in \B_R$,\, $a,b\in R^{n+1}$
\begin{equation*}\label{th-rprez-glo-wl}
\left\{\begin{array}{l}
|f(t,x,a)-f(t,y,b)|\leq 10(n+1)[\,k_R(t)|x-y|+w_R(t,|x-y|)+|a-b|\,]\\[1mm]
|l(t,x,a)-l(t,y,b)|\leq 10(n+1)[\,k_R(t)|x-y|+w_R(t,|x-y|)+|a-b|\,].
\end{array}\right.
\end{equation*}

\noindent Next, if \te{(H4)} is verified, then  for any $t\in[0,T]$, $x\in\R^n$, $a\in\R^{n+1}$
\begin{equation}\label{th-rprez-glo-ww}
|f(t,x,a)|\leq c(t)(1+|x|).
\end{equation}

\noindent Furthermore, if $H$ is continuous, so are $f,l$.
\end{Th}

The proof of Theorem \ref{th-rprez-glo} is contained in Section \ref{pofrepth}. On the basis of the following example we show a simplified version of construction of faithful representation from the above theorem.

\begin{Ex}
Let $H$ and $L$ be as in Example \ref{ex-3}. We know that $H$ satisfies assumptions  (H1)$-$(H4) and \eqref{hip_l1}.  We notice that the set $K:=\E L(0,\cdot)$ is nonempty,
closed and convex, and $K+(0,|x|)=\E L(x,\cdot)$.  Let $P_K:\R^2\to K$ be a projection of the space $\R^2$ to the set  $K$. Obviously, this function fulfills the Lipschitz continuity. We define functions $f,l:\R\times\R^2\to\R$ by the formulas:
\begin{equation*}
f(x,a):=\pi_1\circ P_K(a), \qquad l(x,a):=\pi_2\circ P_K(a)+|x|, \qquad \forall\;(x,a)\in\R\times\R^{2},
\end{equation*}
where $\pi_1(y_1,y_2)=y_1$ and $\pi_2(y_1,y_2)=y_2$ for every $y_1,y_2\in\R$. We show that
the triple $(\R^2,f,l)$ is a representation of the Hamiltonian $H$. Let $e(x,a):=(f(x,a),l(x,a))$ for every $x\in\R$, $a\in\R^2$. Then
$$e(x,a)=(\pi_1\circ P_K(a)\,,\,\pi_1\circ P_K(a))+(0,|x|)=P_K(a)+(0,|x|)$$
for every $x\in\R$, $a\in\R^2$. Therefore,  $e(x,\R^2)=P_K(\R^2)+(0,|x|)=K+(0,|x|)$.  Thus, $e(x,\R^2)=\E L(x,\cdot)=E_L(x)$. It means that the Hamiltonian $H$ and the triple $(\R^2,f,l)$ satisfy assumptions of Proposition \ref{prop-reprezentacja H}, so the triple $(\R^2,f,l)$ has to be a representation of the Hamiltonian $H$.
\end{Ex}

\subsection{Necessary condition of existence of representation $\pmb{(A,f,l)}$ with  compact set $\pmb{A}$}
We start this subsection introducing the condition for an upper bound of the Lagrangian on the effective domain.
\begin{equation}\label{wol1}
\begin{array}{l}
\te{There exists}\;\, \lambda:[0,T]\times\R^n\rightarrow\R\;\, \te{measurable in}\;\, t \;\,\te{for all}\;\, x\in\R^n\;\,\te{and continuous in}\;\,x\;\,\te{for}\\\te{all}\;\,t\in[0,T]\;\, \te{such that}
\;\, L(t,x,v)\leq\lambda(t,x)\;\; \te{for every}\;\;t\in[0,T],\;x\in\R^n,\;v\in\D L(t,x,\cdot),\\
\te{furthermore for any}\; \,R>0\, \;\te{there exists}\;\, w_R(\cdot,\cdot)-\te{modulus}\;\,\te{and}\;\,\mathcal{N}_R-\te{null set such that}\\ |\lambda(t,x)-\lambda(t,y)|\leq w_R(t,|x-y|)\;\,\te{for every}\;\,t\in[0,T]\setminus \mathcal{N}_R\;\,\te{and}\;\,x,y\in\B_R.
\end{array}
\tag{BLC}
\end{equation}

\begin{Th}\label{podr22_th_wk}
Let $A$ be a nonempty compact set. We suppose that  $f:[0,T]\times\R^n\times A\rightarrow\R^n$ and  $l:[0,T]\times\R^n\times A\rightarrow\R$  are  $t-$measurable for all $(x,a)\in\R^n\times\R^m$ and  $(x,a)-$continuous for all $t\in[0,T]$. Further on, we assume that for every $R>0$ there exists the modulus  $w_R(\cdot,\cdot)$ and a null set  $\mathcal{N}_R$ such that $|l(t,x,a)-l(t,y,a)|\leq w_R(t,|x-y|)$ for every  $t\in[0,T]\setminus\mathcal{N}_R$, $x\in\B_R$, $a\in A$. If the Hamiltonian $H$ is given by the formula \eqref{hfl}, then the Lagrangian $L$ given by  \te{(\ref{tran1})} satisfies the condition \eqref{wol1} with the same modulus $w_R(\cdot,\cdot)$. Moreover, if $\,f,l\,$ are continuous, then $\,\lambda\,$ is also continuous.
\end{Th}

The proof of Theorem  \ref{podr22_th_wk} is contained in Section \ref{wk-kon-istr}.
It follows from this Theorem that the condition \eqref{wol1} is necessary condition
for  existence of a continuous and $t-$measurable faithful representation $(A,f,l)$
with the compact control set $A$. We explain  the above theorem on the basis of the following example.

\begin{Ex}\label{ex-6}
We define the Hamiltonian $H:[0,T]\times\R\times\R\rightarrow\R$ by the formula:
\begin{equation*}\label{ex-rh2}
H(t,x,p):=\left\{
\begin{array}{ccl}
|x|\,\max\left\{\,|p|-|\ln t|,0\,\right\} & \te{if} & t\in(0,T], \\[1mm]
0 & \te{if} & t=0.
\end{array}
\right.
\end{equation*}
This Hamiltonian is continuous and satisfies assumptions (H1)$-$(H4) and \eqref{hip_l1}.
From Theorem \ref{th-rprez-glo} there exists a continuous faithful representation $(A,f,l)$ of this Hamiltonian with the control set $A:=\R^2$.\\  The Lagrangian  $L:[0,T]\times\R\times\R\rightarrow\R\cup\{+\infty\}$ given by (\ref{tran1}) has the form
\begin{equation*}\label{ex-rl2}
L(t,x,v)=\left\{
\begin{array}{ccl}
+\infty & \te{if} & v\not\in[-|x|,|x|\,],\;t\in(0,T],\\[1mm]
|\ln t|\,|v| & \te{if} & v\in[-|x|,|x|\,],\;t\in(0,T], \\[1mm]
0 & \te{if} & v=0,\; t=0,\\[1mm]
+\infty & \te{if} & v\not=0,\;t=0.
\end{array}
\right.
\end{equation*}
Obviously, $\D L(t,x,\cdot)=[-|x|,|x|\,]$ for $(t,x)\in(0,T]\times\R$ and $\D L(t,x,\cdot)=0$ for $t=0$, $x\in\R$. The function $t\rightarrow L(t,1,1)=|\ln t|$ on  $(0,T]$ cannot be upper bounded by a continuous function $\lambda(\cdot,\cdot)$ on $[0,T]\times\R$.\linebreak Therefore, by Theorem \ref{podr22_th_wk} we have that for
this Hamiltonian there does not exist a continuous faithful representation $(A,f,l)$,
with the compact control set $A$. We notice that the function $v\rightarrow L(t,x,v)$ is
upper bounded by the function  $\lambda(t,x)=k(t)|x|$  on the set $\D L(t,x,\cdot)$ for
every $(t,x)\in[0,T]\times\R$, where  $k(t):=|\ln t|$ for  $t\in(0,T]$ and $k(t):=0$ for
 $t=0$. Therefore,  $L$ satisfies the condition  \eqref{wol1}. By
Theorem~\ref{th-rprez-glo12} there exists a $t-$measurable faithful representation $(A,f,l)$ of this Hamiltonian with the compact control set $A$. Moreover, a set-valued map $t\to\D L(t,x,\cdot)$ is not continuous in the Hausdorff's sense for all $x\not=0$. Thus,  the set-valued map $t\to E_L(t,x)$ is also discontinuous in the Hausdorff's sense for all $x\not=0$.
\end{Ex}

\begin{Rem}
By Theorem \ref{podr22_th_wk} we have for Hamiltonians from Examples \ref{ex-3}, \ref{ex-4}, \ref{ex-5} that there exist neither continuous nor $t-$measurable faithful representations $(A,\!f\!,l)$ with the compact control set $A$, because functions $v\rightarrow L(t,x,v)$ of these examples are not upper bounded on the effective domain.
Obviously, by Theorem~\ref{th-rprez-glo} there exist continuous faithful
representations  $(A,\!f\!,l)$ of these Hamiltonians  with~the noncompact control set $A$.
\end{Rem}

\subsection{Sufficient condition of existence of representation $\pmb{(A,f,l)}$
with compact set $\pmb{A}$} This subsection is devoted to the new representation theorem of convex Hamiltonians, with the compact control sets.

\begin{Th}[\bf{Representation}]\label{th-rprez-glo12}
We suppose that  $H$ satisfies  \te{(H1)$-$(H4)} and \eqref{hip_l1}. Let $L$ be given by \eqref{tran1} and satisfy \eqref{wol1}. Then there exist  $f:[0,T]\times\R^n\times A\rightarrow\R^n$ and $l:[0,T]\times\R^n\times A\rightarrow\R$, measurable in $t$ for all $(x,a)\in\R^n\times A$ and continuous in $(x,a)$ for all $t\in[0,T]$, with the control set $A:=\B\subset\R^{n+1}$, such that for every $t\in[0,T]$, $x,p\in\R^n$
\begin{equation*}
 H(t,x,p)=\sup_{\;\;a\in \B}\,\{\,\langle\, p,f(t,x,a)\,\rangle-l(t,x,a)\,\}
\end{equation*}
and $f(t,x,A)=\D L(t,x,\cdot)$. Moreover, we have:
\begin{enumerate}
\item[\te{(A1)}] For any $R>0$ and for all $t\in[0,T]\setminus\mathcal{N}_R$,\, $x,y\in \B_R$,\, $a,b\in\B$
\begin{equation*}
\left\{\begin{array}{l}
|f(t,x,a)-f(t,y,a)|\leq 10(n+1)[\,k_R(t)|x-y|+w_R(t,|x-y|)+|M(t,x)-M(t,y)|\,]\\[0mm]
|f(t,x,a)-f(t,x,b)|\leq 10(n+1)M(t,x)|a-b|,
\end{array}\right.
\end{equation*}
where $M(t,x):=|\lambda(t,x)|+|H(t,x,0)|+c(t)(1+|x|)+1$.
\item[\te{(A2)}] $|f(t,x,a)|\leq c(t)(1+|x|)$ for all $t\in[0,T]$, $x\in\R^n$, $a\in\B$.
\item[\te{(A3)}] For any $R>0$ and for all $t\in[0,T]\setminus\mathcal{N}_R$,\, $x,y\in \B_R$,\, $a,b\in\B$
\begin{equation*}
\left\{\begin{array}{l}
|l(t,x,a)-l(t,y,a)|\leq 10(n+1)[\,k_R(t)|x-y|+w_R(t,|x-y|)+|M(t,x)-M(t,y)|\,]\\[0mm]
|l(t,x,a)-l(t,x,b)|\leq 10(n+1)M(t,x)|a-b|,
\end{array}\right.
\end{equation*}
\item[\te{(A4)}] Furthermore, if $H$, $\lambda(\cdot,\cdot)$, $c(\cdot)$ are continuous, so are $f,l$.
\end{enumerate}
\end{Th}

The proof of Theorem \ref{th-rprez-glo12} is contained in Section \ref{pofrepth}.
Now we point out the differences between our construction of faithful representation and the ones contained in
\cite{F-S,FR}. In order to do this, we consider two following examples.

\begin{Ex}\label{prz-rep-010}
Let the Hamiltonian $H$ be as in Example \ref{ex-1}. We know that this Hamiltonian satisfies
assumptions (H1)$-$(H4) and \eqref{hip_l1}.
Our construction of representation $(\hat{A},\hat{f},\hat{l})$ of this Hamiltonian
leads to the set $\hat{A}=[-1,1]\times[-1,1]$ and functions:
\begin{equation*}
\hat{f}(x,a_1,a_2)=a_1|x|, \qquad \hat{l}(x,a_1,a_2)=|a_1|+|a_2|(1-|a_1|),
\end{equation*}
that satisfy the Lipschitz continuity. However, construction of representation  $(\check{A},\check{f},\check{l})$ of this Hamiltonian that is contained in \cite{F-S,FR} leads to the set  $\check{A}=[-1,1]$ and functions:
\begin{equation*}
\check{f}(x,a)=a|x|, \qquad
\check{l}(x,a)=L(x,\check{f}(x,a))=\left\{
\begin{array}{ccl}
|a| & \te{if} & x\not=0 \\[1mm]
0 & \te{if} & x=0.
\end{array}
\right.
\end{equation*}
We notice that the function  $\check{l}$ is discontinuous with respect to the variable $x$ for all $a\in[-1,1]\setminus\{0\}$.
\end{Ex}

\begin{Ex}
Let the Hamiltonian $H$ be as in Example \ref{ex-2}. We know that this Hamiltonian satisfies
assumptions (H1)$-$(H4) and \eqref{hip_l1}. Our construction of representation $(\hat{A},\hat{f},\hat{l})$ of this Hamiltonian
leads to the set $\hat{A}=\{(a_1,a_2)\in\R\times\R\mid a_1^2+a_2^2=1\}$ and functions:
\begin{equation*}
\hat{f}(x,a_1,a_2)=a_1, \qquad \hat{l}(x,a_1,a_2)=a_2+|x|,
\end{equation*}
that satisfy the Lipschitz continuity.  However, construction of representation  $(\check{A},\check{f},\check{l})$ of this Hamiltonian that is contained in \cite{F-S,FR} leads to the set $\check{A}=[-1,1]$ and functions:
\begin{equation*}
\check{f}(x,a)=a, \qquad
\check{l}(x,a)=L(x,\check{f}(x,a))=-\sqrt{1-a^2}+|x|.
\end{equation*}
We notice that the function $\check{l}$ is continuous, but does not satisfy the Lipschitz
continuity with respect to the variable  $a$. Obviously, we have the equalities

\begin{equation*}
\begin{split}
\sup_{(a_1,a_2)\;:\;a_1^2+a_2^2=1}&\,\{\, p\cdot \hat{f}(x,a_1,a_2)-\hat{l}(x,a_1,a_2)\,\}\;\; =\;\; \sup_{(a_1,a_2)\;:\;a_1^2+a_2^2=1}\,\{\, p\cdot a_1-a_2-|x|\,\}\\
&\hspace{-2mm}= \sup_{a_1\in[-1,1]}\,\left\{\, p\cdot a_1+\sqrt{1-a_1^2}-|x|\,\right\}\;\;=\;\; \sup_{a\in[-1,1]}\,\{\, p\cdot \check{f}(x,a)-\check{l}(x,a)\,\}.
\end{split}
\end{equation*}
\end{Ex}

\subsection{Stability of representations}\label{subsecstarep}
In this section we will see that the faithful representation obtained in the previous section is stable.

\begin{Th}\label{thm-rep-stab1}
Let $H_i,H:[0,T]\times\R^n\times\R^n\to\R$, $i\in\N$ be continuous and satisfy \te{(H1)$-$(H3)}, \eqref{hip_l1}. We consider the representations $(\R^{n+1}\!\!,f_i,l_i)$ and $(\R^{n+1}\!\!,f,l)$ of $H_i$ and $H$, respectively, defined as in the proof of Theorem~\ref{th-rprez-glo}. If $H_i$ converge uniformly on compacts to $H$, then $f_i$ converge to $f$ and $l_i$ converge to $l$ uniformly on compacts in $[0,T]\times\R^n\times\R^{n+1}$.
Furthermore, if $ H_i, H, i\in\N$ satisfy \te{(H4)}, then $ f_i, f, i\in\N$ satisfy~\eqref{th-rprez-glo-ww}.
\end{Th}

\begin{Th}\label{thm-rep-stab2}
Let $H_i,H:[0,T]\times\R^n\times\R^n\to\R$, $i\in\N$ satisfy \te{(H1)$-$(H4)}, \eqref{hip_l1}. We suppose that $L_i$, $L$ , $i\in\N$ are given by \eqref{tran1} and satisfy \eqref{wol1}. Let $H_i,\lambda_i,c_i$, $i\in\N$ be continuous. We consider the representations $(\B,f_i,l_i)$ and $(\B,f,l)$ of $H_i$ and $H$, respectively,  defined as in the proof of Theorem~\ref{th-rprez-glo12}. If $H_i,\lambda_i,c_i$ converge uniformly on compacts to $H,\lambda,c$, then $f_i$ converge to $f$ and $l_i$ converge to $l$ uniformly on compacts in $[0,T]\times\R^n\times\B$.
\end{Th}

\begin{Th}\label{thm-rep-stab3}
Let $H_i,H:[0,T]\times\R^n\times\R^n\to\R$, $i\in\N$ satisfy \te{(H1)$-$(H3)}, \eqref{hip_l1}. We consider the representations $(\R^{n+1}\!\!,f_i,l_i)$ and $(\R^{n+1}\!\!,f,l)$ of $H_i$ and $H$, respectively, defined as in the proof of Theorem~\ref{th-rprez-glo}. If $H_i(t,\cdot,\cdot)$ converge uniformly on compacts to $H(t,\cdot,\cdot)$ for all $t\in[0,T]$, then $f_i(t,\cdot,\cdot)$ converge to $f(t,\cdot,\cdot)$ and $l_i(t,\cdot,\cdot)$ converge to $l(t,\cdot,\cdot)$ uniformly on compacts in $\R^n\times\R^{n+1}$ for all $t\in[0,T]$.
Furthermore, if $ H_i, H, i\in\N$ satisfy \te{(H4)}, then $ f_i, f, i\in\N$ satisfy \eqref{th-rprez-glo-ww}.
\end{Th}

\begin{Th}\label{thm-rep-stab4}
Let $H_i,H:[0,T]\times\R^n\times\R^n\to\R$, $i\in\N$ satisfy \te{(H1)$-$(H4)}, \eqref{hip_l1}. We suppose that $L_i$, $L$, $i\in\N$ are given by \eqref{tran1} and satisfy \eqref{wol1}. We consider the representations $(\B,f_i,l_i)$ and $(\B,f,l)$ of $H_i$ and $H$, respectively, defined as in the proof of Theorem~\ref{th-rprez-glo12}. If $H_i(t,\cdot,\cdot),\;\lambda_i(t,\cdot)$ converge uniformly on compacts to $H(t,\cdot,\cdot),\;\lambda(t,\cdot)$, and $c_i(t)\to c(t)$ for all $t\in[0,T]$, then $f_i(t,\cdot,\cdot)$ converge to $f(t,\cdot,\cdot)$ and $l_i(t,\cdot,\cdot)$ converge to $l(t,\cdot,\cdot)$ uniformly on compacts in $\R^n\times\B$ for all $t\in[0,T]$.
\end{Th}

The proofs of Theorems \ref{thm-rep-stab1}, \ref{thm-rep-stab2}, \ref{thm-rep-stab3}, \ref{thm-rep-stab4} are contained in Section \ref{thms-stab}.

\begin{Th}\label{thm-rep-stab5}
Let $H_i,H:[0,T]\times\R^n\times\R^n\to\R$, $i\in\N$ be continuous and satisfy \te{(H1)$-$(H3)}, \eqref{hip_l1}. We consider the representations $(\R^{n+1}\!\!,f_i,l_i)$ and $(\R^{n+1}\!\!,f,l)$ of $H_i$ and $H$, respectively,  defined as in the proof of Theorem~\ref{th-rprez-glo}. If $H_i$ and  $H$ satisfy the condition
\begin{equation}\label{thm-rep-stab5-11}
\sup_{(t,x,p)\,\in\,[0,T]\times\B_R\times\R^n}\frac{|H_i(t,x,p)-H(t,x,p)|}{1+|p|}\;\xrightarrow[i\to\infty]{}\;0,\quad\forall\;R>0,
\end{equation}
then $f_i,f$ and $l_i,l$ satisfy the conditions
\begin{equation*}
\left\{\begin{array}{l}
\sup\limits_{(t,x,a)\,\in\,[0,T]\times\B_R\times\R^{n+1}}|f_i(t,x,a)-f(t,x,a)|\;\xrightarrow[i\to\infty]{}\;0,\quad\forall\;R>0,\\[3mm]
\sup\limits_{(t,x,a)\,\in\,[0,T]\times\B_R\times\R^{n+1}}\;\;|l_i(t,x,a)-l(t,x,a)|\;\xrightarrow[i\to\infty]{}\;0,\quad\forall\;R>0.
\end{array}\right.
\end{equation*}
 Furthermore, if $ H_i, H, i\in\N$ satisfy \te{(H4)}, then $ f_i, f, i\in\N$ satisfy \eqref{th-rprez-glo-ww}.
\end{Th}

\begin{Th}\label{thm-rep-stab6}
Let $H_i,H:[0,T]\times\R^n\times\R^n\to\R$, $i\in\N$ satisfy \te{(H1)$-$(H3)}, \eqref{hip_l1}. We consider the representations $(\R^{n+1}\!\!,f_i,l_i)$ and $(\R^{n+1}\!\!,f,l)$ of $H_i$ and $H$, respectively,  defined as in the proof of Theorem~\ref{th-rprez-glo}. If $H_i$ and  $H$ satisfy the condition
\begin{equation}
\sup_{(x,p)\,\in\,\B_R\times\R^n}\frac{|H_i(t,x,p)-H(t,x,p)|}{1+|p|}\;\xrightarrow[i\to\infty]{}\;0,\quad\forall\;t\in[0,T],\;\forall\;R>0,
\end{equation}
then $f_i,f$ and $l_i,l$ satisfy the conditions
\begin{equation*}
\left\{\begin{array}{l}
\sup\limits_{(x,a)\,\in\,\B_R\times\R^{n+1}}|f_i(t,x,a)-f(t,x,a)|\;\xrightarrow[i\to\infty]{}\;0,\quad\forall\;t\in[0,T],\;\forall\;R>0,\\[3mm]
\sup\limits_{(x,a)\,\in\,\B_R\times\R^{n+1}}\;\;|l_i(t,x,a)-l(t,x,a)|\;\xrightarrow[i\to\infty]{}\;0,\quad\forall\;t\in[0,T],\;\forall\;R>0.
\end{array}\right.
\end{equation*}
 Furthermore, if $ H_i, H, i\in\N$ satisfy \te{(H4)}, then $ f_i, f, i\in\N$ satisfy \eqref{th-rprez-glo-ww}.
\end{Th}

The proofs of Theorems \ref{thm-rep-stab5}, \ref{thm-rep-stab6} are also contained in
Section \ref{thms-stab}.


\section{Proof of Theorem \ref{podr22_th_wk}}\label{wk-kon-istr}

\noindent Before we prove Theorem \ref{podr22_th_wk} we state and prove three auxiliary lemmas.

\begin{Lem}\label{lem-tran-r1}
We suppose that the set $A$ is nonempty and compact. Let  $f:[0,T]\times\R^n\times A\rightarrow\R^n$ and\linebreak $l:[0,T]\times\R^n\times A\rightarrow\R$ be $a-$continuous functions for every $t\in[0,T]$, $x\in\R^n$. If  $H:[0,T]\times\R^n\times\R^n\rightarrow\R$ is given by the formula
\begin{equation*}
 H(t,x,p):= \sup_{a\in A}\,\{\,\langle\, p,f(t,x,a)\,\rangle-l(t,x,a)\,\},
\end{equation*}
then for all $t\in[0,T]$, $x\in\R^n$, $a\in A$ we have $L(t,x,f(t,x,a))\leq l(t,x,a)$, where $L(t,x,\cdot\,):=H^{\ast}(t,x,\cdot\,)$.
\end{Lem}

\begin{proof}
We assume, by contradiction, that the assertion is false. Then there exist  $t\in[0,T]$, $x\in\R^n$, $\kr{a}\in A$ such that  $l(t,x,\kr{a})<L(t,x,f(t,x,\kr{a}))$. Therefore $(f(t,x,\kr{a}), l(t,x,\kr{a}))\not\in\E L(t,x,\cdot)$. The function
$p\rightarrow H(t,x,p)$ is finite and convex, so by \cite[Thm. 11.1]{R-W} the function $v\rightarrow L(t,x,v)$ is proper, convex, lower semicontinuous and  $H(t,x,\cdot\,)=L^{\ast}(t,x,\cdot\,)$. Hence the set $\E L(t,x,\cdot)$
is nonempty, closed and convex. By Epigraph Separation Theorem, there exists $\kr{p}\in\R^n$ such that
\begin{equation}\label{stod-dlae}
\sup_{(v,\eta)\in\E L(t,x,\cdot)} \langle\, (v,\eta),(\kr{p},-1)\,\rangle \;<\; \langle\, (f(t,x,\kr{a}), l(t,x,\kr{a})),(\kr{p},-1)\,\rangle.
\end{equation}
If  $v\in\D L(t,x,\cdot)$, then $(v,L(t,x,v))\in\E L(t,x,\cdot)$. Thus, by the inequality
 (\ref{stod-dlae}) and the equality $H(t,x,\cdot\,)=L^{\ast}(t,x,\cdot\,)$ we obtain
\begin{eqnarray*}
H(t,x,\kr{p}) &=& \sup_{v\in\mathbb{R}^{n}}\,\{\,\langle \kr{p},v\rangle-L(t,x,v)\,\}\;=\;\sup_{v\in\D L(t,x,\cdot)}\,\{\,\langle \kr{p},v\rangle-L(t,x,v)\,\}\\
&=&\sup_{v\in\D L(t,x,\cdot)}\langle (v,L(t,x,v)),(\kr{p},-1)\rangle\;\leq\; \sup_{(v,\eta)\in\E L(t,x,\cdot)} \langle\, (v,\eta),(\kr{p},-1)\,\rangle\\
&<& \langle\, (f(t,x,\kr{a}), l(t,x,\kr{a})),(\kr{p},-1)\,\rangle\;=\; \langle\, \kr{p},f(t,x,\kr{a})\,\rangle-l(t,x,\kr{a})\\
&\leq & H(t,x,\kr{p}).
\end{eqnarray*}
Thus, we obtain a contradiction that ends the proof.
\end{proof}

\begin{Lem}\label{lem-tran-r2}
We assume that the set $A$ is nonempty and compact. Let  $f:[0,T]\times\R^n\times A\rightarrow\R^n$ and $l:[0,T]\times\R^n\times A\rightarrow\R$ be an $a-$continuous function
for every $t\in[0,T]$, $x\in\R^n$. We assume that the set $f(t,x,A)$ is convex for every
$t\in[0,T]$, $x\in\R^n$. If  $H:[0,T]\times\R^n\times\R^n\rightarrow\R$
is given by the formula
\begin{equation*}
 H(t,x,p):= \sup_{a\in A}\,\{\,\langle\, p,f(t,x,a)\,\rangle-l(t,x,a)\,\},
\end{equation*}
then for all $t\in[0,T]$, $x\in\R^n$ we have $f(t,x,A)=\D L(t,x,\cdot)$, where $L(t,x,\cdot\,):=H^{\ast}(t,x,\cdot\,)$.
\end{Lem}

\begin{proof}
We put $t\in[0,T]$ and $x\in\R^n$. The function $p\rightarrow H(t,x,p)$ is finite and convex,
so by \cite[Thm. 11.1]{R-W}, the function $v\rightarrow L(t,x,v)$ is proper, convex and
lower semicontinuous. By Lemma \ref{lem-tran-r1} we have the inequality $L(t,x,f(t,x,a))\leq l(t,x,a)$ for every  $a\in A$. Thus, by the property of the function
$v\rightarrow L(t,x,v)$ we obtain $f(t,x,A)\subset \D L(t,x,\cdot)$. Now we show that $\D L(t,x,\cdot)\subset f(t,x,A)$. We suppose that this inclusion is false.
Then there exist $\kr{v}\in \D L(t,x,\cdot)$ and $\kr{v}\not\in f(t,x,A)$. The set $f(t,x,A)$ is nonempty, convex and compact, so by the Separation Theorem,
there exist $\kr{p}\in\R^n$ and numbers $\alpha,\beta\in\R$ such that
\begin{equation*}
 \langle\,\kr{v},\kr{p} \,\rangle\leq \alpha<\beta\leq \langle\, f(t,x,a),\kr{p}\,\rangle,  \;\;\;\forall\;a\in A.
\end{equation*}
We notice that by the above inequality we obtain the inequality
\begin{equation}\label{stod-dlae0101}
 \beta-\alpha\;\leq\; \langle\, f(t,x,a)-\kr{v},\kr{p}\,\rangle,  \;\;\;\forall\;a\in A.
\end{equation}
We set $\xi(t,x):=\inf_{a\in A}l(t,x,a)$. Let $\kr{n}\in N$ be large enough that the following inequality holds
\begin{equation}\label{stod-dlae01011}
L(t,x,\kr{v})-\xi(t,x)\;<\;\kr{n}\cdot(\beta-\alpha).
\end{equation}
Using assumptions, we get for $\,\kr{q}:=-\kr{n}\cdot\kr{p}\,$  the existence of  $\,a_{\kr{q}}\in A\,$ with
\begin{eqnarray}\label{stod-dlae01012}
\nonumber H(t,x,\kr{q}) &=&\sup_{a\in A}\,\{\,\langle\, \kr{q},f(t,x,a)\,\rangle-l(t,x,a)\,\}\\
&=& \langle\, \kr{q},f(t,x,a_{\kr{q}})\,\rangle-l(t,x,a_{\kr{q}}).
\end{eqnarray}
Putting together the inequalities  (\ref{stod-dlae01011}), (\ref{stod-dlae01012}) and (\ref{stod-dlae0101}), we obtain
\begin{eqnarray*}
\kr{n}\cdot(\beta-\alpha) &>& L(t,x,\kr{v})-\xi(t,x)\\
&=& \sup_{p\in\mathbb{R}^{n}}\,\{\,\langle \kr{v},p\rangle-H(t,x,p)\,\}-\xi(t,x)\\
&\geq & \langle\,\kr{v},\kr{q}\,\rangle - H(t,x,\kr{q})-\xi(t,x)\\
&= & \langle\,\kr{v},\kr{q}\,\rangle -\langle\, \kr{q},f(t,x,a_{\kr{q}})\,\rangle +l(t,x,a_{\kr{q}})-\xi(t,x)\\
&\geq & \langle\, \kr{v}-f(t,x,a_{\kr{q}}),\kr{q}\,\rangle\\
&=& \langle\, \kr{v}-f(t,x,a_{\kr{q}}),-\kr{n}\cdot\kr{p}\,\rangle\\
&=& \kr{n}\cdot\langle\, f(t,x,a_{\kr{q}})-\kr{v},\kr{p}\,\rangle\\
&\geq & \kr{n}\cdot(\beta-\alpha).
\end{eqnarray*}
Thus, we obtain a contradiction that ends the proof.
\end{proof}

\begin{Lem}\label{lem-tran-r1w1}
We suppose that the set $A$ is nonempty and compact. Let  $f:[0,T]\times\R^n\times A\rightarrow\R^n$ and $l:[0,T]\times\R^n\times A\rightarrow\R$ be $t-$measurable
functions for all $(x,a)\in\R^n\times A$ and $(x,a)-$continuous for all $t\in[0,T]$.
If  $H:[0,T]\times\R^n\times\R^n\rightarrow\R$ is given by the formula
\begin{equation*}
 H(t,x,p):= \sup_{a\in A}\,\{\,\langle\, p,f(t,x,a)\,\rangle-l(t,x,a)\,\},
\end{equation*}
then there exist a nonempty, compact set $\mathbbmtt{A}$ and functions $\mathbbmtt{f}:[0,T]\times\R^n\times \mathbbmtt{A}\rightarrow\R^n$ and $\mathbbmtt{l}:[0,T]\times\R^n\times \mathbbmtt{A}\rightarrow\R$ measurable in $t$ for all $(x,\mathbbmtt{a})\in\R^n\times \mathbbmtt{A}$ and continuous in $(x,\mathbbmtt{a})$ for all $t\in[0,T]$ such that for every $t\in[0,T]$, $x\in\R^n$, $p\in\R^n$ we have
\begin{equation}\label{rwr-lem010}
 H(t,x,p)=\sup_{\mathbbmtt{a}\in \mathbbmtt{A}}\,\{\,\langle\, p,\mathbbmtt{f}(t,x,\mathbbmtt{a})\,\rangle-\mathbbmtt{l}(t,x,\mathbbmtt{a})\,\}.
\end{equation}

Besides, for all $t\in[0,T]$, $x\in\R^n$
\begin{eqnarray}\label{rwr-lem0}
\mathbbmtt{f}(t,x,\mathbbmtt{A})=\te{conv}f(t,x,A),\qquad \mathbbmtt{l}(t,x,\mathbbmtt{A})=\te{conv}\,l(t,x,A),
\end{eqnarray}

If for any $R>0$ there exists modulus $w_R(\cdot,\cdot)$ and null set $\mathcal{N}_R$ such that $|l(t,x,a)-l(t,y,a)|\leq w_R(t,|x-y|)$ for all $t\in[0,T]\setminus\mathcal{N}_R$, $x\in\B_R$, $a\in A$,\, then\, $|\mathbbmtt{l}(t,x,\mathbbmtt{a})-\mathbbmtt{l}(t,y,\mathbbmtt{a})|\leq w_R(t,|x-y|)$ for every $t\in[0,T]\setminus\mathcal{N}_R$, $x\in\B_R$, $\mathbbmtt{a}\in \mathbbmtt{A}$ and $R>0$.

Besides, if functions  $f,l$ are continuous, then functions $\mathbbmtt{f},\mathbbmtt{l}$
are also continuous.
\end{Lem}

\begin{proof}
We define a simplex in the space $\R^{n+1}$ by
$$\Delta:= \{(\alpha_1,\dots,\alpha_{n+1})\in[0,1]^{n+1}\mid\alpha_1+\dots+\alpha_{n+1}=1\}.$$
Obviously, the set $\Delta$ is compact. Moreover, we define the set $\,\mathbbmtt{A}\,$
by $\mathbbmtt{A}:=A^{n+1}\times\Delta$. We notice that $\mathbbmtt{A}$
is a compact subset of the space $\R^{2n+2}$. The functions $\mathbbmtt{f}$, $\mathbbmtt{l}$ are defined for every $t\in[0,T]$, $x\in\R^n$ and $\mathbbmtt{a}=(a_1,\dots,a_{n+1},\alpha_1,\dots,\alpha_{n+1})\in A^{n+1}\times\Delta=\mathbbmtt{A}$ by formulas:
$$\mathbbmtt{f}(t,x,\mathbbmtt{a}):= \sum_{i=1}^{n+1}\alpha_i\, f(t,x,a_i),\qquad \mathbbmtt{l}(t,x,\mathbbmtt{a}):= \sum_{i=1}^{n+1}\alpha_i\, l(t,x,a_i).$$
We notice that $\mathbbmtt{f}$, $\mathbbmtt{l}$ are  $t-$measurable for all $(x,\mathbbmtt{a})\in\R^n\times \mathbbmtt{A}$ and $(x,\mathbbmtt{a})-$continuous for all $t\in[0,T]$. Besides, if functions $f,l$ are continuous, then functions $\mathbbmtt{f},\mathbbmtt{l}$ are also continuous.

Now we prove that the triple  $(\mathbbmtt{A},\mathbbmtt{f},\mathbbmtt{l})$ satisfies
the equality \eqref{rwr-lem010}. To this purpose, we fix  $t\in[0,T]$, $x\in\R^n$ and $p\in\R^n$. Let $\widehat{\mathbbmtt{a}}:=(a,\dots,a,\alpha_1,\dots,\alpha_{n+1})\in \mathbbmtt{A}$.
Then by the definition of the triple $(\mathbbmtt{A},\mathbbmtt{f},\mathbbmtt{l})$ we
obtain
$$\mathbbmtt{f}(t,x,\widehat{\mathbbmtt{a}})= f(t,x,a),\qquad \mathbbmtt{l}(t,x,\widehat{\mathbbmtt{a}})= l(t,x,a).$$
By the above equalities
\begin{eqnarray*}
\nonumber \langle\, p,f(t,x,a)\,\rangle-l(t,x,a)  &=& \langle\, p,\mathbbmtt{f}(t,x,\widehat{\mathbbmtt{a}})\,\rangle -\mathbbmtt{l}(t,x,\mathbbmtt{\widehat{a}})\\
&\leq & \sup_{\mathbbmtt{a}\;\in\; \mathbbmtt{A}}\,\{\,\langle\, p,\mathbbmtt{f}(t,x,\mathbbmtt{a})\,\rangle -\mathbbmtt{l}(t,x,\mathbbmtt{a})\,\}
\end{eqnarray*}
for all   $a\in A$. Therefore we get the inequality
\begin{equation}\label{rwr-lem1}
H(t,x,p)\;\leq\;  \sup_{\mathbbmtt{a}\;\in\; \mathbbmtt{A}}\,\{\,\langle\, p,\mathbbmtt{f}(t,x,\mathbbmtt{a})\,\rangle -\mathbbmtt{l}(t,x,\mathbbmtt{a})\,\}.
\end{equation}
On the other hand, for every $\mathbbmtt{a}\in \mathbbmtt{A}$
\begin{eqnarray*}
\langle\, p,\mathbbmtt{f}(t,x,\mathbbmtt{a})\,\rangle -\mathbbmtt{l}(t,x,\mathbbmtt{a}) &= &\sum_{i=1}^{n+1}\alpha_i[\langle\, p,f(t,x,a_i)\,\rangle -l(t,x,a_i)] \\
&\leq & \sum_{i=1}^{n+1}\alpha_i H(t,x,p)\;=\;H(t,x,p).
\end{eqnarray*}
It means that the following inequality holds
\begin{eqnarray}\label{rwr-lem2}
\sup_{\mathbbmtt{a}\in \mathbbmtt{A}}\,\{\,\langle\, p,\mathbbmtt{f}(t,x,\mathbbmtt{a})\,\rangle -\mathbbmtt{l}(t,x,\mathbbmtt{a})\,\}\;\leq\;H(t,x,p).
\end{eqnarray}
Combining the inequality (\ref{rwr-lem1}) and (\ref{rwr-lem2}) we obtain  \eqref{rwr-lem010}.

The equality  (\ref{rwr-lem0})  follows easily from the definition of the triple
$(\mathbbmtt{A},\mathbbmtt{f},\mathbbmtt{l})$ and Caratheodory's Theorem (see \cite[Thm. 2.29]{R-W}).

Let $|l(t,x,a)-l(t,y,a)|\leq w_R(t,|x-y|)$ for every $t\in[0,T]\setminus\mathcal{N}_R$, $x\in\B_R$, $a\in A$, $R>0$. Then from the definition of $\mathbbmtt{l}$ we have
\begin{eqnarray*}
|\mathbbmtt{l}(t,x,\mathbbmtt{a})-\mathbbmtt{l}(t,y,\mathbbmtt{a})| &\leq &  \sum_{i=1}^{n+1}\alpha_i\,|l(t,x,a_i)-l(t,y,a_i)|\\
&\leq &  \sum_{i=1}^{n+1}\alpha_i\,w_R(t,|x-y|)\;\;=\;\;w_R(t,|x-y|)
\end{eqnarray*}
for all $t\in[0,T]\setminus\mathcal{N}_R$, $x\in\B_R$, $\mathbbmtt{a}\in \mathbbmtt{A}$, $R>0$.
\end{proof}

\begin{proof}[Proof of Theorem \ref{podr22_th_wk}]
By Lemma  \ref{lem-tran-r1w1} there exist a nonempty, compact set $\mathbbmtt{A}$ and
functions $\mathbbmtt{f}$, $\mathbbmtt{l}$ measurable in $t$ for all $(x,\mathbbmtt{a})\in\R^n\times \mathbbmtt{A}$ and continuous in $(x,\mathbbmtt{a})$ for all $t\in[0,T]$ such that the triple $(\mathbbmtt{A},\mathbbmtt{f},\mathbbmtt{l})$ is a representation of $H$ and $\mathbbmtt{f}(t,x,\mathbbmtt{A})=\te{conv}f(t,x,A)$ for every $t\in[0,T]$, $x\in\R^n$. Therefore, by Lemma \ref{lem-tran-r2} we have for all $t\in[0,T]$, $x\in\R^n$
\begin{eqnarray}\label{podr221}
\mathbbmtt{f}(t,x,\mathbbmtt{A})=\D L(t,x,\cdot).
\end{eqnarray}

Now, we prove that the condition \eqref{wol1} holds. Let $\lambda(t,x):=\sup_{\mathbbmtt{a}\in \mathbbmtt{A}}\mathbbmtt{l}(t,x,\mathbbmtt{a})$. Obviously, the function $\lambda$ is $t-$measurable for all $x\in\R^n$ and  $x-$continuous for all $t\in[0,T]$.

Set $t\in[0,T]$ and $x\in\R^n$. If $\kr{v}\in\D L(t,x,\cdot)$ then by the equality (\ref{podr221}), there exists $\kr{\mathbbmtt{a}}\in \mathbbmtt{A}$ such that   $\kr{v}=\mathbbmtt{f}(t,x,\kr{\mathbbmtt{a}})$. Therefore by Lemma \ref{lem-tran-r1}
 $$L(t,x,\kr{v})=L(t,x,\mathbbmtt{f}(t,x,\kr{\mathbbmtt{a}}))\leq \mathbbmtt{l}(t,x,\kr{\mathbbmtt{a}})\leq \lambda(t,x).$$
It means that  $L(t,x,v)\leq \lambda(t,x)$ for every  $t\in[0,T]$, $x\in\R^n$, $v\in\D L(t,x,\cdot)$.

Let $|l(t,x,a)-l(t,y,a)|\leq w_R(t,|x-y|)$ for all $t\in[0,T]\setminus\mathcal{N}_R$, $x\in\B_R$, $a\in A$, $R>0$. Then by Lemma~\ref{lem-tran-r1w1} we have  $|\mathbbmtt{l}(t,x,\mathbbmtt{a})-\mathbbmtt{l}(t,y,\mathbbmtt{a})|\leq w_R(t,|x-y|)$ for every $t\in[0,T]\setminus\mathcal{N}_R$, $x\in\B_R$, $\mathbbmtt{a}\in \mathbbmtt{A}$, $R>0$. We set $t\in[0,T]\setminus\mathcal{N}_R$, $x\in\B_R$ and $R>0$.  Let $\kr{\mathbbmtt{a}}\in \mathbbmtt{A}$ be such that  $\lambda(t,x)=\mathbbmtt{l}(t,x,\kr{\mathbbmtt{a}})$. Then
\begin{eqnarray*}
\lambda(t,x)-\lambda(t,y) &= & \mathbbmtt{l}(t,x,\kr{\mathbbmtt{a}})-\sup_{\mathbbmtt{a}\in \mathbbmtt{A}}\mathbbmtt{l}(t,y,\mathbbmtt{a})\\ & \leq &  \mathbbmtt{l}(t,x,\kr{\mathbbmtt{a}})-\mathbbmtt{l}(t,y,\kr{\mathbbmtt{a}})
\;\;\leq \;\; w_R(t,|x-y|).
\end{eqnarray*}
In addition to this, $t\in[0,T]\setminus\mathcal{N}_R$, $x\in\B_R$ and $R>0$ are arbitrary, so we have
$|\lambda(t,x)-\lambda(t,y)|\leq w_R(t,|x-y|)$ for every $t\in[0,T]\setminus\mathcal{N}_R$, $x\in\B_R$ and $R>0$.

Besides, if functions $f,l$ are continuous, then by Lemma \ref{lem-tran-r1w1}, the functions $\mathbbmtt{f},\mathbbmtt{l}$ are also continuous. Therefore, the function $\lambda$
has to be continuous.
\end{proof}


\section{Proofs of representation theorems}\label{pofrepth}\label{section-5}

\noindent First, we propose some auxiliary definitions and facts. By $\mathcal{P}_{fc}(\R^m)$
we denote a family of all nonempty, closed and convex subsets of $\R^m$. Then, let  $\mathcal{P}_{kc}(\R^m)$ be a family of all nonempty, convex and compact subsets of $\R^m$.

\begin{Lem}[\te{\cite[p. 369]{A-F}}]\label{lem-pmh}
The set-valued map $P:\R^m\times\mathcal{P}_{fc}(\R^m)\multimap\mathcal{P}_{kc}(\R^m)$ defined by
\begin{equation*}
P(y,K):= K\cap\B(y,2d(y,K))
\end{equation*}
is Lipschitz with the Lipschitz constant $5$, i.e. for all $K,D\in\mathcal{P}_{fc}(\R^m)$ and $x,y\in\R^m$
\begin{equation*}
\mathscr{H}(P(x,K),P(y,D))\leq 5(\mathscr{H}(K,D)+|x-y|).
\end{equation*}
\end{Lem}

The support function $\sigma(K,\cdot):\R^m\to\R$ of the set $K\in\mathcal{P}_{kc}(\R^m)$ is a convex function defined by
\begin{equation*}
\sigma(K,p):=\max_{x\in K}\,\langle p,x\rangle,\quad \forall\,p\in\R^m.
\end{equation*}

\noindent Being Lipschitz, $\sigma(K,\cdot)$ is differentiable a.e. in $\R^m$. Let $m(\noo\sigma(K,p))$ be the element of $\noo\sigma(K,p)$ with the minimal norm. It coincides with $\nabla\sigma(K,p)$ at every $p\in\R^m$ where $\sigma(K,\cdot)$ is differentiable.

\begin{Def}\label{df-scel}
For any $K\in\mathcal{P}_{kc}(\R^m)$,  its Steiner point is defined by
\begin{equation*}
s_m(K):=\frac{1}{\mathrm{vol}(\B)}\int_{\B}m(\noo\sigma(K,p))\,dp\in K,
\end{equation*}
where $\mathrm{vol}(\B)$ is the measure of the $m-$dimensional unit ball $\B\subset\R^m$.
\end{Def}

\begin{Lem}[\te{\cite[p. 366]{A-F}}]\label{lem-scmh}
The function $s_m(\cdot)$ is  Lipschitz in the Hausdorff metric with the Lipschitz constant $m$ on the set of all nonempty convex compact subsets of $\R^m$, i.e.
\begin{equation*}
|s_m(K)-s_m(D)|\leq m\, \mathscr{H}(K,D),\quad \forall\,K,D\in\mathcal{P}_{kc}(\R^m).
\end{equation*}
\end{Lem}

Let a set-valued map $E:[0,T]\times\R^n\multimap\R^m$ have nonempty, closed values. If a
set-valued map $t\to E(t,x)$ is measurable for every $x\in\R^n$, then the single-valued map $t\to d(y,E(t,x))$ is measurable for every $x\in\R^n$, $y\in\R^m$ (see \cite[Thm. 14.3]{R-W}). If a set-valued map $x\to E(t,x)$  is lower semicontinuous and has a closed graph for every
$t\in[0,T]$, then a  single-valued map $x\to d(y,E(t,x))$ is continuous for every
$t\in[0,T]$, $y\in\R^m$ (see \cite[Prop. 5.11]{R-W}). Besides, the inequality $|d(y,E(t,x))-d(z,E(t,x))|\leq|y-z|$ holds for every $t\in[0,T]$, $x\in\R^n$, $y,z\in\R^m$. Thus, we obtain the following proposition:

\begin{Prop}\label{dow-prp}
We assume that a set-valued map $E:[0,T]\times\R^n\multimap\R^m$ has nonempty, closed values, $E(\cdot,x)$ is measurable for every $x\in\R^n$ and $E(t,\cdot)$, has a closed
graph and is lower semicontinuous for every  $t\in[0,T]$. If a single-valued map $M:[0,T]\times\R^n\to \R$ is $t-$measurable for each $x\in\R^n$ and $x-$continuous for
every $t\in[0,T]$, then a single-valued map defined by
\begin{equation*}
 (t,x,a)\to d(M(t,x)\,a,E(t,x)),\quad \forall\,(t,x,a)\in[0,T]\times\R^n\times\R^m
\end{equation*}
is $t-$measurable for every $(x,a)\in\R^n\times\R^m$ and $(x,a)-$continuous for every  $t\in[0,T]$. In addition to this, it is a $(t,x,a)-$continuous map, if $M$ is continuous, $E$ has
a closed graph and is lower semicontinuous.
\end{Prop}

\begin{Def}
Let $d_H(K,D):= \sup_{z\in K}d(z,D)$. We say a  set-valued map $F$ is \it{upper} (respectively, \it{lower}) \it{semicontinuous in the sense of the Hausdorff's distance}, what we denote by $\mathscr{H}-$usc ($\mathscr{H}-$lsc), if for every point $x_0$ and arbitrary
number $\varepsilon>0$, there exists a number $\delta>0$ such that for every $x\in\B(x_0,\delta)$ the condition  $d_H(F(x_0),F(x))<\varepsilon$ (resp. $d_H(F(x),F(x_0))<\varepsilon$) holds.
\end{Def}
\noindent
Obviously, a set-valued map $F$ is continuous in the sense of the  Hausdorff's distance ($\mathscr{H}-$continuous) if and only if it is $\mathscr{H}-$usc and $\mathscr{H}-$lsc.

The Hausdorff's distance between closed balls can be estimated in the following way:
\begin{equation}\label{dow-omk}
\mathscr{H}(B(x,r),B(y,s))\leq |x-y|+|r-s|,\;\; \forall x,y\in\R^n,\;\; \forall\,r,s\geq 0.
\end{equation}

\begin{Th}\label{th-oparam}
We suppose a set-valued map $E:[0,T]\times\R^n\multimap\R^m$ has nonempty, closed and
convex values, $E(\cdot,x)$ is measurable for every $x\in\R^n$, $E(t,\cdot)$ has a closed graph and is lower semicontinuous for every $t\in[0,T]$. Let a single-valued map $M:[0,T]\times\R^n\to \R_+$ be $t-$measurable for every $x\in\R^n$ and $x-$continuous
for every $t\in[0,T]$. Then there exists a single-valued map $e:[0,T]\times\R^n\times\R^m\to\R^m$ such that $e(\cdot,x,a)$ is measurable for
every  $x\in\R^n$, $a\in\R^m$ and $e(t,\cdot,\cdot)$ is continuous  for every $t\in[0,T]$.
Besides, for every $t\in[0,T]$, $x,y\in\R^n$, $a,b\in\R^m$ it satisfies the equality
$e(t,x,\R^m)=E(t,x)$ and the inequality
\begin{equation}\label{th-oparam-n}
|e(t,x,a)-e(t,y,b)|\leq 5m[\,\mathscr{H}(E(t,x),E(t,y))+|M(t,x)\,a-M(t,y)\,b|\,].
\end{equation}

Additionally, a single-valued map $e$ is continuous, if $M$ is continuous and  $E$ has a closed graph, and is lower semicontinuous.

If a set-valued map $Q:[0,T]\times\R^n\multimap\R^m$ for every $t\in[0,T]$, $x\in\R^n$ satisfies $Q(t,x)\subset E(t,x)$ and $\|Q(t,x)\|\leq M(t,x)$, then for every $t\in[0,T]$, $x\in\R^n$ we have $Q(t,x)\subset e(t,x,\B)\subset E(t,x)$.
\end{Th}

Theorem \ref{th-oparam} is a version of Parametrization Theorems 9.7.1 and 9.7.2
from the monograph \cite{A-F}.

\begin{proof}[Proof of Theorem \ref{th-oparam}]
Let $(t,x,a)\in[0,T]\times\R^n\times\R^m$. We consider the closed ball $G(t,x,a)\subset \R^m$ of center $M(t,x)\,a$ and radius $2d(M(t,x)\,a,E(t,x))$, i.e.
$$G(t,x,a):=\B(M(t,x)\,a,2d(M(t,x)\,a,E(t,x))).$$
 By the inequality  \eqref{dow-omk}, Proposition \ref{dow-prp} and \cite[Cor. 8.2.13]{A-F} a set-valued map $G(\cdot,x,a)$ is measurable for every  $x\in\R^n$, $a\in\R^m$ and a  set-valued map $G(t,\cdot,\cdot)$ is $\mathscr{H}-$continuous for every $t\in[0,T]$.
Besides, $\|G(t,x,a)\|\leq \varphi(t,x,a)$ for every $t\in[0,T]$, $x\in\R^n$, $a\in\R^m$, where
$$\varphi(t,x,a):= M(t,x)\,|a|+2d(M(t,x)\,a,E(t,x)).$$
By Proposition \ref{dow-prp} and the hypotheses, we obtain  $\varphi(\cdot,x,a)$
is measurable  for all $x\in\R^n$, $a\in\R^m$ and $\varphi(t,\cdot,\cdot)$ is continuous
for all $t\in[0,T]$.

Let $P$ be the map defined in Lemma \ref{lem-pmh}. We set
$$\Phi(t,x,a):= P(M(t,x)\,a,E(t,x))=E(t,x)\cap G(t,x,a).$$
By Corollary \ref{wrow-wm} and hypotheses, the set $\Phi(t,x,a)$ is nonempty, compact
 and convex.  The maps $G(\cdot,x,a)$ and $E(\cdot,a)$ are measurable and have closed values,
so the map $\Phi(\cdot,x,a)$ being their intersection is also measurable for all $x\in\R^n$, $a\in\R^m$ (see \cite[Thm. 8.2.4]{A-F}).
Now we show that a  map $\Phi(t,\cdot,\cdot)$ is $\mathscr{H}-$continuous for all $t\in[0,T]$. Indeed, setting $t\in[0,T]$, the  map $\Phi(t,\cdot,\cdot)$ has a closed graph, because it is an intersection of maps $G(t,\cdot,\cdot)$ and $E(t,\cdot)$ having
closed graphs.
Moreover, $\|\Phi(t,x,a)\|\leq \varphi(t,x,a)$ for all $x\in\R^n$, $a\in\R^m$. By continuity of $\varphi(t,\cdot,\cdot)$ we have $\Phi(t,\cdot,\cdot)$ is locally bounded.
It means that  $\Phi(t,\cdot,\cdot)$ is $\mathscr{H}-$usc. Thus, we are to
prove $\Phi(t,\cdot,\cdot)$ is $\mathscr{H}-$lsc. To do this, it is suffices to show that
it is lower semicontinuous in the Kuratowski's sense, because it has compact values. We fix $(x,a)\in\R^n\times\R^m$ and the open set $O\subset\R^m$ such that $\Phi(t,x,a)\cap O\not=\emptyset$.
If $\I G(t,\cdot,\cdot)=\emptyset$, then $G(t,x,a)\subset O$. We know a map $G(t,\cdot,\cdot)$ has compact values and is $\mathscr{H}-$continuous. Thus, we have
$G(t,x',a')\subset O$ for all $(x',a')$ near $(x,a)$. Therefore $\Phi(t,x',a')\subset G(t,x',a')\subset O$ for all $(x',a')$ near $(x,a)$.
Let  $z_1\in \Phi(t,x,a)\cap O$ and $\I G(t,\cdot,\cdot)\not=\emptyset$. Then by the definition of $G(t,\cdot,\cdot)$ there exists $z_2\in E(t,x)\cap \I G(t,x,a)$. Thus, the interval $(z_1,z_2]\subset E(t,x)\cap \I G(t,x,a)$. Consequently, we can find an element $z\in\R^m$ satisfying $z\in O\cap E(t,x)\cap \I G(t,x,a)$. Hence, for some $\varepsilon>0$ we have $\B(z,\varepsilon)\subset G(t,x,a)\cap O$. The set-valued map $G(t,\cdot, \cdot)$ is a ball whose center and radius are continuous functions. Hence, for every $(x',a')$ sufficiently close to $(x,a)$ we have $\B(z,\varepsilon/2)\subset G(t,x',a')$. On the other hand, $E$ is lower semicontinuous, so $\B(z,\varepsilon/2)\cap E(t,x')\not=\emptyset$ for all $x'$ near $x$. Therefore for every $(x',a')$ sufficiently close to $(x,a)$ we have $\Phi(t,x',a')\cap O\not=\emptyset$. Thus, the  map $\Phi(t,\cdot,\cdot)$ is lower
semicontinuous.

We define the single-valued map \,$e$\, from \,$[0,T]\times\R^n\times\R^m$\, to \,$\R^m$\, by
$$ e(t,x,a):= s_m(\Phi(t,x,a)),$$
where $s_m$ in the Steiner selection. Since $\Phi$ is measurable with respect to $t$, using the equivalent of the definition of $s_m$ from \cite[p. 365]{A-F}, we deduce that $e$ is also measurable with respect to $t$. by Lemma \ref{lem-scmh} we have for all $t,s\in[0,T]$, $x,y\in\R^n$, $a,b\in\R^m$,
\begin{equation}\label{dow-nreh}
|e(t,x,a)-e(s,y,b)|\leq m\mathscr{H}(\Phi(t,x,a),\Phi(s,y,b)).
\end{equation}
We have shown that  $\Phi(t,\cdot,\cdot)$ is $\mathscr{H}-$continuous for every $t\in[0,T]$. By the inequality \eqref{dow-nreh} we have $e(t,\cdot,\cdot)$ is continuous
for every $t\in[0,T]$. Additionally, if $E$ has a closed graph and is lower
semicontinuous, and  $M$ is continuous, then similarly to the above, one can prove
that    $\Phi$ is $\mathscr{H}-$continuous. Then by the inequality
 \eqref{dow-nreh} we have that a single-valued map  $e$ is continuous.

We notice that by the inequality \eqref{dow-nreh} and Lemma \ref{lem-pmh}  for all $t\in[0,T]$, $x,y\in\R^n$, $a,b\in\R^m$ we obtain
\begin{equation*}
|e(t,x,a)-e(t,y,b)|\leq 5m[\,\mathscr{H}(E(t,x),E(t,y))+|M(t,x)\,a-M(t,y)\,b|\,].
\end{equation*}

Now we show that $e(t,x,\R^m)=E(t,x)$. For this purpose, we fix $t\in[0,T]$, $x\in\R^n$ and $z\in E(t,x)$. Setting
$$a:=z/M(t,x)$$
we derive
$$a\in\R^m,\quad M(t,x)\,a=z,\quad \Phi(t,x,a)=\{z\}.$$
the above and Definition \ref{df-scel} imply that
$$e(t,x,a)=s_m(\Phi(t,x,a))=z.$$
This means that $E(t,x)\subset e(t,x,\R^m)$ for all $t\in[0,T]$, $x\in\R^n$. The opposite inclusion is the result of Definition~\ref{df-scel}, i.e. $e(t,x,a)=s_m(\Phi(t,x,a))\in \Phi(t,x,a)\subset E(t,x)$ for all $t\in[0,T]$, $x\in\R^n$, $a\in\R^m$.

Let the set-valued map $Q:[0,T]\times\R^n\multimap\R^m$ for every $t\in[0,T]$, $x\in\R^n$ satisfies $Q(t,x)\subset E(t,x)$ and $\|Q(t,x)\|\leq M(t,x)$. We prove that $Q(t,x)\subset e(t,x,\B)\subset E(t,x)$ for every $t\in[0,T]$, $x\in\R^n$. For this purpose, we fix $t\in[0,T]$, $x\in\R^n$ and $z\in Q(t,x)$. Using the hypotheses, we have
$z\in E(t,x)$ and $a:=z/M(t,x)\in\B$. Therefore $\Phi(t,x,a)=\{z\}$. By Definition \ref{df-scel}
we get $e(t,x,a)=s_m(\Phi(t,x,a))=z$. This means that $Q(t,x)\subset e(t,x,\B)$ for all $t\in[0,T]$, $x\in\R^n$. The second inclusion is the result of $e(t,x,\B)\subset e(t,x,\R^m)=E(t,x)$.
\end{proof}

\subsection{Proof of Theorem \ref{th-rprez-glo}} This subsection is devoted to the proof of the new representation theorem for convex Hamiltonians, with noncompact control sets.

\begin{Prop}\label{prop-reprezentacja H}
We suppose that a function $p\rightarrow H(t,x,p)$ is proper, convex and lower semicontinuous.
Let $e(t,x,A)=E_L(t,x)$ and $L(t,x,\cdot\,)=H^{\ast}(t,x,\cdot\,)$.
If $e(t,x,a)=(f(t,x,a),l(t,x,a))$ for every $a\in A$, then
\begin{equation}\label{prop-reprezentacja H-row}
 H(t,x,p)=\sup_{a\in A}\,\{\,\langle\, p,f(t,x,a)\,\rangle-l(t,x,a)\,\}.
\end{equation}
Besides,  $f(t,x,A)=\D L(t,x,\cdot)$ and $l(t,x,a)\geq -|H(t,x,0)|$ for every $a\in A$.
\end{Prop}

\begin{proof}
We know that $e(t,x,a)\in E_L(t,x)$ for every $a\in A$, so $(f(t,x,a),l(t,x,a))\in E_L(t,x)$ for every  $a\in A$. Therefore, by the definition of the set $E_L(t,x)$ we obtain $L(t,x, f(t,x,a))\leq l(t,x,a)$ for every $a\in A$. Thus, $f(t,x,a)\in\D L(t,x,\cdot)$ for every $a\in A$, because the function  $v\rightarrow L(t,x,v)$ is proper (see \cite[Thm. 11.1]{R-W}). Therefore, for every  $a\in A$ we have
\begin{eqnarray*}
\langle p,f(t,x,a)\rangle-l(t,x,a) &\leq & \langle\, p,f(t,x,a)\,\rangle-L(t,x,f(t,x,a))\\
&\leq &\;\;\; \sup_{v\in\D L(t,x,\cdot)}\,\{\,\langle p,v\rangle-L(t,x,v)\,\}\;\;=\;\;H(t,x,p).\\
\end{eqnarray*}
By the above argumentation $f(t,x,A)\subset\D L(t,x,\cdot)$ and
\begin{equation}\label{nrp1}
\sup_{a\in A}\,\{\,\langle\, p,f(t,x,a)\,\rangle-l(t,x,a)\,\}\leq H(t,x,p).
\end{equation}

We fix  $\kr{v}\in\D L(t,x,\cdot)$. We know that $(\kr{v},L(t,x,\kr{v}))\in E_L(t,x)=e(t,x,A)$, so there exists $\kr{a}\in A$ such that  $(\kr{v},L(t,x,\kr{v}))=e(t,x,\kr{a})=(f(t,x,\kr{a}),l(t,x,\kr{a}))$. Therefore,  $\kr{v}=f(t,x,\kr{a})$ and $L(t,x,\kr{v})=l(t,x,\kr{a})$. Moreover,
\begin{eqnarray*}
\langle p,\kr{v}\rangle-L(t,x,\kr{v}) &= & \langle\, p,f(t,x,\kr{a})\,\rangle-l(t,x,\kr{a})\\
&\leq & \sup_{a\in A}\,\{\,\langle\, p,f(t,x,a)\,\rangle-l(t,x,a)\,\}.
\end{eqnarray*}
So $\D L(t,x,\cdot)\subset f(t,x,A)$ and
\begin{equation*}
\sup_{v\in\D L(x,\cdot)}\,\{\,\langle p,v\rangle-L(t,x,v)\,\}\leq \sup_{a\in A}\,\{\,\langle\, p,f(t,x,a)\,\rangle-l(t,x,a)\,\}.
\end{equation*}
By the last inequality and the equality $H(t,x,\cdot\,)=L^{\ast}(t,x,\cdot\,)$, that
holds because of \cite[Thm. 11.1]{R-W}, we obtain
\begin{eqnarray}\label{nrp2}
\nonumber H(t,x,p) &= & \sup_{v\in\D L(t,x,\cdot)}\,\{\,\langle p,v\rangle-L(t,x,v)\,\}\\
&\leq & \sup_{a\in A}\,\{\,\langle\, p,f(t,x,a)\,\rangle-l(t,x,a)\,\}.
\end{eqnarray}

Putting together inequalities \eqref{nrp1}, \eqref{nrp2} we get the equality
\eqref{prop-reprezentacja H-row}. Besides,  $f(t,x,A)=\D L(t,x,\cdot)$. Finally, we notice that
by the equality $L(t,x,\cdot\,)=H^{\ast}(t,x,\cdot\,)$ for all $v\in\R^n$ we have
$L(t,x,v)\geq -|H(t,x,0)|$. Thus,
$l(t,x,a)\geq L(t,x,f(t,x,a))\geq-|H(t,x,0)|$ for all $a\in A$.
\end{proof}

\begin{Th}\label{do-th-parame-glo1}
We assume that  $H$ satisfies  \te{(H1)$-$(H3)} and \eqref{hip_l1}. Let $L$ be given by \eqref{tran1}. Then there exists a map  $e:[0,T]\times\R^n\times\R^{n+1}\rightarrow\R^{n+1}$, measurable in $t$ for all $(x,a)\in\R^n\times\R^{n+1}$ and continuous in $(x,a)$ for all $t\in[0,T]$ such that for every $t\in[0,T]$ and $x\in\R^n$ it satisfies
\begin{equation}\label{do-th-parame-glo1-r1}
e\left(t,\,x,\,\R^{n+1}\right)=E_L(t,x).
\end{equation}
Moreover, for any $R>0$ and for all $t\in[0,T]\setminus\mathcal{N}_R$,\, $x,y\in \B_R$,\, $a,b\in R^{n+1}$
\begin{equation}\label{do-th-parame-glo1-r2}
|e(t,x,a)-e(t,y,b)|\leq 10(n+1)[\,k_R(t)|x-y|+w_R(t,|x-y|)+|a-b|\,].
\end{equation}
Furthermore, if $H$ is continuous, so is $e$.
\end{Th}

\begin{proof}
Let $M(t,x)\equiv 1$ and $E(t,x):=E_L(t,x)$ for every $t\in[0,T]$, $x\in\R^n$.
Using hypotheses and Corollary~\ref{wrow-wm} we have the maps $M$ and $E$ satisfy
assumptions of  Theorem \ref{th-oparam}. Therefore,  there exists a map  $e:[0,T]\times\R^n\times\R^{n+1}\rightarrow\R^{n+1}$ measurable in $t$ for all $(x,a)\in\R^n\times\R^{n+1}$ and continuous in $(x,a)$ for all $t\in[0,T]$ such that for every $t\in[0,T]$, $x,y\in\R^n$, $a,b\in\R^{n+1}$ it satisfies the equality  \eqref{do-th-parame-glo1-r1} and the inequality \eqref{th-oparam-n}. By the inequality
\eqref{th-oparam-n} and Corollary~\ref{hlc-cor-ner} we have
\begin{eqnarray*}
|e(t,x,a)-e(t,y,b)| &\leq & 5(n+1)[\,\mathscr{H}(E_L(t,x),E_L(t,y))+|a-b|\,]\\
&\leq & 10(n+1)[\,k_R(t)|x-y|+w_R(t,|x-y|)\,]+5(n+1)|a-b|,
\end{eqnarray*}
for all $t\in[0,T]\setminus\mathcal{N}_R$,\, $x,y\in \B_R$,\, $a,b\in R^{n+1}$ and $R>0$.
It means that the inequality \eqref{do-th-parame-glo1-r2} holds. Additionally, if
we assume that Hamiltonian $H$ is continuous, then using hypotheses, Corollary \ref{wrow-wm} and  Theorem~\ref{th-oparam} we obtain that the map $e$ is continuous.
\end{proof}

\begin{Rem}
Let $e:[0,T]\times\R^n\times\R^{n+1}\rightarrow\R^{n+1}$ be a function from Theorem \ref{do-th-parame-glo1}. We define two functions $f:[0,T]\times\R^n\times \R^{n+1}\rightarrow\R^n$ and $l:[0,T]\times\R^n\times \R^{n+1}\rightarrow\R$ by formulas:
\begin{equation*}
f(t,x,a):=\pi_v(e(t,x,a))\;\;\;\te{and}\;\;\;l(t,x,a):=\pi_\eta(e(t,x,a)),
\end{equation*}
where $\pi_v(v,\eta)=v$ and $\pi_\eta(v,\eta)=\eta$ for every $v\in\R^n$ and $\eta\in\R$. Then for every $t\in[0,T]$, $x\in\R^n$, $a\in \R^{n+1}$ the following equality holds
\begin{equation*}
e(t,x,a)=(f(t,x,a),l(t,x,a)).
\end{equation*}
Thus, for every $t\in[0,T]$, $x,y\in\R^n$, $a,b\in \R^{n+1}$ we obtain
\begin{equation*}
|f(t,x,a)-f(t,y,b)| \;\leq\;  |e(t,x,a)-e(t,y,b)|,\quad |l(t,x,a)-l(t,y,b)|\; \leq \; |e(t,x,a)-e(t,y,b)|.
\end{equation*}
It means, by Corollary \ref{wrow-wm}, Proposition \ref{prop-reprezentacja H} and Theorem \ref{do-th-parame-glo1}, that functions $f,l$ satisfy every condition asserted in Theorem \ref{th-rprez-glo}.
\end{Rem}

\subsection{Proof of Theorem \ref{th-rprez-glo12}} This subsection is devoted to the proof of the new representation theorem for convex Hamiltonians, with compact control sets.

\begin{Prop}\label{prop-reprezentacja H-ogr}
We suppose that the function $p\rightarrow H(t,x,p)$ is proper, convex and  lower semicontinuous.
Let $E_{\lambda,L}(t,x)\subset e(t,x,\B)\subset E_L(t,x)$ and $L(t,x,\cdot\,)=H^{\ast}(t,x,\cdot\,)$.
If $L(t,x,v)\leq\lambda(t,x)$ for all $v\in\D L(t,x,\cdot)$ and $e(t,x,a)=(f(t,x,a),l(t,x,a))$ for all $a\in \B$, then
\begin{equation}\label{prop-reprezentacja H-row-org}
 H(t,x,p)=\sup_{a\in \B}\,\{\,\langle\, p,f(t,x,a)\,\rangle-l(t,x,a)\,\}.
\end{equation}
Moreover, $f(t,x,\B)=\D L(t,x,\cdot)$.
\end{Prop}

\begin{proof}
Because $e(t,x,a)\in E_L(t,x)$ for every  $a\in \B$, it follows that $(f(t,x,a),l(t,x,a))\in E_L(t,x)$ for
every  $a\in \B$. Therefore by the definition of the set $E_L(t,x)$, we obtain $L(t,x, f(t,x,a))\leq l(t,x,a)$ for every $a\in \B$. Hence $f(t,x,a)\in\D L(t,x,\cdot)$ for all  $a\!\in\! \B$, because the function  $v\!\rightarrow\! L(t,x,v)$ is proper (see \cite[Thm. 11.1]{R-W}). Thus, for every  $a\in \B$ we have
\begin{eqnarray*}
\langle p,f(t,x,a)\rangle-l(t,x,a) &\leq & \langle\, p,f(t,x,a)\,\rangle-L(t,x,f(t,x,a))\\
&\leq &\;\;\; \sup_{v\in\D L(t,x,\cdot)}\,\{\,\langle p,v\rangle-L(t,x,v)\,\}\;\;=\;\;H(t,x,p).\\
\end{eqnarray*}
Thus $f(t,x,\B)\subset\D L(t,x,\cdot)$ and
\begin{equation}\label{nrp1-org}
\sup_{a\in \B}\,\{\,\langle\, p,f(t,x,a)\,\rangle-l(t,x,a)\,\}\leq H(t,x,p).
\end{equation}

We set  $\kr{v}\in\D L(t,x,\cdot)$. Using assumptions, $L(t,x,\kr{v})\leq\lambda(t,x)$. Therefore $(\kr{v},L(t,x,\kr{v}))\in E_{\lambda,L}(t,x)\subset e(t,x,\B)$. So, there exists
$\kr{a}\in \B$ such that $(\kr{v},L(t,x,\kr{v}))=e(t,x,\kr{a})=(f(t,x,\kr{a}),l(t,x,\kr{a}))$. Hence $\kr{v}=f(t,x,\kr{a})$ and $L(t,x,\kr{v})=l(t,x,\kr{a})$. Besides,
\begin{eqnarray*}
\langle p,\kr{v}\rangle-L(t,x,\kr{v}) &= & \langle\, p,f(t,x,\kr{a})\,\rangle-l(t,x,\kr{a})\\
&\leq & \sup_{a\in \B}\,\{\,\langle\, p,f(t,x,a)\,\rangle-l(t,x,a)\,\}.
\end{eqnarray*}
So $\D L(t,x,\cdot)\subset f(t,x,\B)$ and
\begin{equation*}
\sup_{v\in\D L(x,\cdot)}\,\{\,\langle p,v\rangle-L(t,x,v)\,\}\leq \sup_{a\in \B}\,\{\,\langle\, p,f(t,x,a)\,\rangle-l(t,x,a)\,\}.
\end{equation*}
By the last inequality and the equality  $H(t,x,\cdot\,)=L^{\ast}(t,x,\cdot\,)$, that holds because of~\cite[Thm. 11.1]{R-W},  we obtain
\begin{eqnarray}\label{nrp2-org}
\nonumber H(t,x,p) &= & \sup_{v\in\D L(t,x,\cdot)}\,\{\,\langle p,v\rangle-L(t,x,v)\,\}\\
&\leq & \sup_{a\in \B}\,\{\,\langle\, p,f(t,x,a)\,\rangle-l(t,x,a)\,\}.
\end{eqnarray}

Combining inequalities (\ref{nrp1-org}) and (\ref{nrp2-org}) we obtain the equality
\eqref{prop-reprezentacja H-row-org}. Additionally, we have that  $f(t,x,\B)=\D L(t,x,\cdot)$.
\end{proof}

\begin{Th}\label{do-th-parame-lo1}
We Assume that  $H$ satisfies  \te{(H1)$-$(H4)} and \eqref{hip_l1}. Let $L$ be given by \eqref{tran1} and satisfy \eqref{wol1}. Then there exists  $e:[0,T]\times\R^n\times\B\rightarrow\R^{n+1}$, where $\B\subset\R^{n+1}$, measurable in $t$ for all $(x,a)\in\R^n\times\B$ and continuous in $(x,a)$ for all $t\in[0,T]$ such that for every $t\in[0,T]$, $x\in\R^n$
\begin{equation}\label{do-th-parame-lo1-r1}
E_{\lambda,L}(t,x)\subset e\left(t,x,\B\right)\subset E_L(t,x).
\end{equation}
Moreover, for any $R>0$ and for all $t\in[0,T]\setminus\mathcal{N}_R$,\, $x,y\in \B_R$,\, $a,b\in \B$
\begin{equation}\label{do-th-parame-lo1-r2}
\left\{\begin{array}{l}
|e(t,x,a)-e(t,y,b)|\leq 10(n+1)[\,k_R(t)|x-y|+w_R(t,|x-y|)+|M(t,x)\,a-M(t,y)\,b|\,],\\[0mm]
\it{where}\;\;M(t,x):=|\lambda(t,x)|+|H(t,x,0)|+c(t)(1+|x|)+1.
\end{array}\right.
\end{equation}
Furthermore, if $H$, $\lambda(\cdot,\cdot)$, $c(\cdot)$ are continuous, so is $e$.
\end{Th}

\begin{proof}
Let $M(t,x):=|\lambda(t,x)|+|H(t,x,0)|+c(t)(1+|x|)+1$ and $E(t,x):=E_L(t,x)$ for every
$t\in[0,T]$, $x\in\R^n$. Using hypotheses and Corollary \ref{wrow-wm} we have the
maps $M$ and $E$ satisfy conditions of Theorem \ref{th-oparam}. Therefore, there exists a map $e:[0,T]\times\R^n\times\B\rightarrow\R^{n+1}$, where $\B\subset\R^{n+1}$, measurable in $t$ for all $(x,a)\in\R^n\times\B$ and continuous in $(x,a)$ for all $t\in[0,T]$ such that for every $t\in[0,T]$, $x,y\in\R^n$, $a,b\in\B$ it satisfies the inequality
\eqref{th-oparam-n}. By the inequality \eqref{th-oparam-n} and Corollary~\ref{hlc-cor-ner} we have
\begin{eqnarray*}
|e(t,x,a)-e(t,y,b)| &\leq & 5(n+1)[\,\mathscr{H}(E_L(t,x),E_L(t,y))+|M(t,x)\,a-M(t,y)\,b|\,]\\
&\leq & 10(n+1)[\,k_R(t)|x-y|+w_R(t,|x-y|)\,]+5(n+1)|M(t,x)a-M(t,y)b|
\end{eqnarray*}
for all $t\in[0,T]\setminus\mathcal{N}_R$,\, $x,y\in \B_R$,\, $a,b\in \B$ and $R>0$.
It means that the inequality  \eqref{do-th-parame-lo1-r2} is satisfied. Additionally, if
we assume that $H$, $\lambda(\cdot,\cdot)$, $c(\cdot)$ are continuous, then the map $M$ is
also continuous. Using hypotheses, Corollary \ref{wrow-wm} and  Theorem \ref{th-oparam} we have the map $e$ is continuous.

Let $Q(t,x):=E_{\lambda,L}(t,x)$ for all $t\in[0,T]$, $x\in\R^n$. We prove that  $Q$
satisfies the assumptions of  Theorem~\ref{th-oparam}. Obviously, $E_{\lambda,L}(t,x)\subset E_L(t,x)$ for every  $t\in[0,T]$, $x\in\R^n$. Thus, it suffices to show that
$\|E_{\lambda,L}(t,x)\|\leq M(t,x)$ for every $t\in[0,T]$, $x\in\R^n$. We fix $t\in[0,T]$, $x\in\R^n$ and  $(v,\eta)\in E_{\lambda,L}(t,x)$. Then by the definition of the set
$E_{\lambda,L}(t,x)$, we have that $L(t,x,v)\leq\eta\leq\lambda(t,x)$. Therefore, using assumptions and the equality $L(t,x,\cdot\,)=H^{\ast}(t,x,\cdot\,)$, we obtain
$$v\in\D L(t,x,\cdot)\;\;\;\te{and}\;\;\;-|H(t,x,0)|\leq\eta\leq\lambda(t,x).$$
By the point (M5) in Corollary \ref{wrow-wm} we have $\|\D L(t,x,\cdot)\|\leq c(t)(1+|x|)$. So, $|v|\leq c(t)(1+|x|)$ and $|\eta|\leq|\lambda(t,x)|+|H(t,x,0)|$. Therefore, $|(v,\eta)|\leq |\lambda(t,x)|+|H(t,x,0)|+c(t)(1+|x|)< M(t,x)$.
It means that a set-valued map $Q$ satisfies the hypotheses of Theorem \ref{th-oparam}.
By this Theorem, the condition  \eqref{do-th-parame-lo1-r1} is satisfied.
\end{proof}

\begin{Rem}
Let $e:[0,T]\times\R^n\times\B\rightarrow\R^{n+1}$ be the function from Theorem \ref{do-th-parame-lo1}. We define two functions $f:[0,T]\times\R^n\times \B\rightarrow\R^n$ and $l:[0,T]\times\R^n\times \B\rightarrow\R$ by formulas:
\begin{equation*}
f(t,x,a):=\pi_v(e(t,x,a))\;\;\;\te{and}\;\;\;l(t,x,a):=\pi_\eta(e(t,x,a)),
\end{equation*}
where  $\pi_v(v,\eta)=v$ and $\pi_\eta(v,\eta)=\eta$ for every $v\in\R^n$ and $\eta\in\R$. Then for every $t\in[0,T]$, $x\in\R^n$, $a\in \B$ the following equality holds
\begin{equation*}
e(t,x,a)=(f(t,x,a),l(t,x,a)).
\end{equation*}
So, for every $t\in[0,T]$, $x,y\in\R^n$, $a,b\in \B$ we obtain
\begin{equation*}
|f(t,x,a)-f(t,y,b)| \;\leq\;  |e(t,x,a)-e(t,y,b)|,\quad |l(t,x,a)-l(t,y,b)|\; \leq \; |e(t,x,a)-e(t,y,b)|.
\end{equation*}
By Corollary \ref{wrow-wm}, Proposition \ref{prop-reprezentacja H-ogr} and Theorem \ref{do-th-parame-lo1}, it means that functions $f,l$ satisfy all conditions all the assertions of Theorem \ref{th-rprez-glo12}.
\end{Rem}



\section{Proofs of stability theorems}\label{thms-stab}

\noindent We show here that the faithful representation obtained in this paper is stable. To do this,
we need a few auxiliary definitions and facts. For a sequence $\{K_i\}_{i\in\N}$ of subsets of $\R^m$, the \it{upper limit} is the set
\begin{equation*}
\limsup_{i\to\infty}K_i:=\{\,x\in\R^m\,\mid\,\te{there exists}\;x_{j}\to x\;\te{such that}\;x_{j}\in K_{i_j}\;\te{for all large}\;j\in\N\, \},
\end{equation*}
while the  \it{lower limit} is the set
\begin{equation*}
\hspace{-3.5mm}\liminf_{i\to\infty}K_i:=\{\,x\in\R^m\,\mid\,\te{there exists}\;x_{i}\to x\;\te{such that}\;x_{i}\in K_{i}\;\te{for all large}\;i\in\N\, \}.
\end{equation*}
The \it{limit} of a sequence exists if the upper and lower limit sets are equal:
\begin{equation*}
\lim_{i\to\infty}K_i:= \limsup_{i\to\infty}K_i=\liminf_{i\to\infty}K_i.
\end{equation*}
For nonempty, closed sets $K_i$ and $K$, one has $\lim_{i\to\infty}K_i=K$  if and only if $\lim_{i\to\infty}d(x,K_i)=d(x,K)$ for every $x\in\R^m$ (see \cite[Cor. 4.7]{R-W}).
Thus, using the inequality $|d(x,K)-d(y,K)|\leq|x-y|$, that is satisfied for every
$x,y\in\R^m$ and every nonempty set $K\subset\R^m$, we obtain
 \begin{equation}\label{zmkno}
\lim_{i\to\infty}x_i=x,\quad\lim_{i\to\infty}K_i=K\quad\Longrightarrow\quad\lim_{i\to\infty}d(x_i,K_i)=d(x,K).
\end{equation}
If $K_i$ and $K$ are nonempty, closed subsets of a given compact set, then by \cite[Chap. 4, Sec C.]{R-W}
\begin{equation}\label{zmkh}
\lim_{i\to\infty}K_i=K\;\iff\; \lim_{i\to\infty}\mathscr{H}(K_i,K)=0.
\end{equation}

\begin{Th}[\te{\cite[Thm. 4.32]{R-W}}]\label{thmozkp}
Let  $K_i$ and $D_i$ be convex sets in $\R^m$ for all $i\in\N$. If convex sets $K$ and $D$ satisfy $K\cap \I D\neq\emptyset$, then the following implication holds: $$\lim_{i\to\infty}K_i=K,\quad \lim_{i\to\infty}D_i=D\quad\Longrightarrow\quad \lim_{i\to\infty}\left(\,K_i\cap D_i\,\right)=K\cap D.$$
\end{Th}

\vspace*{-5mm}
\pagebreak

\subsection{Convergence of set-valued maps $\pmb{E_{L_i}(\cdot,\cdot)}$}
In this subsection we show that convergence of  Hamiltonians $H_i$ implies
convergence of set-valued maps $E_{L_i}(\cdot,\cdot)$.

\begin{Prop}\label{epiconv}
Let $H_i,H:[0,T]\times\R^n\times\R^n\to\R$, $i\in\N$ be continuous and satisfy \te{(H3)}. We suppose that $L_i$, $L$, $i\in\N$ are given by \eqref{tran1}.  If $H_i$ converge to $H$ uniformly on compacts in $[0,T]\times\R^n\times\R^n$, then
\begin{enumerate}
\item[\textnormal{(i)}] $\liminf\limits_{i\to\infty} L_i(t_i,x_i,v_i)\geq L(t,x,v)$\; for every sequence\; $(t_i,x_i,v_i)\rightarrow (t,x,v)$,
\item[\textnormal{(ii)}] $\forall\,(t,x,v)\in[0,T]\times\R^n\times\R^n\;\;\forall\,(t_i,x_i)\rightarrow (t,x)\;\;\exists\,v_i\rightarrow v\;:\;L_i(t_i,x_i,v_i)\rightarrow L(t,x,v)$.
\end{enumerate}
\end{Prop}

\noindent Proposition \ref{epiconv} is a consequence of Wijsman's Theorem \cite[Thm. 11.34]{R-W}.

\begin{Prop}\label{klconv}
Let $H_i,H:[0,T]\times\R^n\times\R^n\to\R$, $i\in\N$ be continuous and satisfy \te{(H3)}. We assume $L_i$, $L$, $i\in\N$ are given by \eqref{tran1}.  If $H_i$ converge to $H$ uniformly on compacts in $[0,T]\times\R^n\times\R^n$, then
\begin{equation*}
\lim_{i\to\infty}E_{L_i}(t_i,x_i)=E_L(t,x)\;\;\it{for every sequence}\;\;(t_i,x_i)\to(t,x).
\end{equation*}
\end{Prop}

\begin{proof}
First, we prove that
\begin{equation}\label{in-gk-1}
\limsup_{i\to\infty}E_{L_i}(t_i,x_i)\subset E_L(t,x),\quad \forall\;(t_i,x_i)\to(t,x).
\end{equation}
We suppose that $(t_i,x_i)\to(t,x)$. Let $(v,\eta)\in\limsup_{i\to\infty}E_{L_i}(t_i,x_i)$.
By the definition of upper limit sets, there exists the sequence $(v_j,\eta_j)\to (v,\eta)$ such that  $(v_j,\eta_j)\in E_{L_{i_j}}(t_{i_j},x_{i_j})$ for all large $j\in\N$. Hence we have $L_{i_j}(t_{i_j},x_{i_j},v_{j})\leq \eta_j$ for all large $j\in\N$. By the point (i) of Proposition \ref{epiconv},
\begin{equation*}
L(t,x,v)\leq \liminf\limits_{j\to\infty}L_{i_j}(t_{i_j},x_{i_j},v_{j})\leq \lim\limits_{j\to\infty}\eta_j=\eta.
\end{equation*}
Therefore $(v,\eta)\in E_L(t,x)$, so the inclusion \eqref{in-gk-1} is true.

Thus, the equality of our Proposition is true if we prove that
\begin{equation}\label{in-gk-2}
 E_L(t,x)\subset\liminf_{i\to\infty}E_{L_i}(t_i,x_i),\quad \forall\;(t_i,x_i)\to(t,x).
\end{equation}
We assume $(t_i,x_i)\to(t,x)$. Let $(v,\eta)\in E_L(t,x)$. Then $L(t,x,v)\leq \eta$. Therefore the value $L(t,x,v)$ is finite because the function $v\to L(t,x,v)$ is proper. By the point~(ii) of Proposition \ref{epiconv}, there exists a sequence $v_i\to v$ such that $L_i(t_i,x_i,v_i)\to L(t,x,v)$.
Therefore for large $i\in\N$, the values $L_i(t_i,x_i,v_i)$ have to be finite because
the value $L(t,x,v)$ is finite. We notice that
\begin{equation*}
L_i(t_i,x_i,v_i)\leq L_i(t_i,x_i,v_i)+\eta-L(t,x,v)=:\eta_i.
\end{equation*}
Therefore ($v_i,\eta_i)\to (v,\eta)$ and $(v_i,\eta_i)\in E_{L_i}(t_i,x_i)$ for large $i\in\N$. It means that $(v,\eta)$ belongs to the set $\liminf_{i\to\infty}E_{L_i}(t_i,x_i)$, so the proof of inclusion \eqref{in-gk-2} is over.
\end{proof}

\begin{Prop}\label{gklconv}
Let $H_i,H:[0,T]\times\R^n\times\R^n\to\R$, $i\in\N$ be continuous and satisfy \te{(H3)}. We suppose that $L_i$, $L$, $i\in\N$ are given by \eqref{tran1}.  If $H_i$ and  $H$ satisfy the condition \eqref{thm-rep-stab5-11}, then
\begin{equation*}
    \sup_{(t,x)\,\in\,[0,T]\times\B_R}\mathscr{H}(E_{L_i}(t,x),E_{L}(t,x))\;
        \xrightarrow[i\to\infty]{}\;0,\quad\forall\;R>0.
\end{equation*}
\end{Prop}

\noindent Proposition \ref{gklconv} can be proven similarly like Propositions \ref{0rwwl1} and \ref{0lnrwwl1}.

\subsection{Proofs of stability theorems}
Let $H_i,H:[0,T]\times\R^n\times\R^n\to\R$, $i\in\N$ be continuous and satisfy \te{(H3)}. We assume $L_i$, $L$ , $i\in\N$ are given by \eqref{tran1}. We consider continuous single-valued maps $M_i,M:[0,T]\times\R^n\to\R_+$, $i\in\N$.

Let $(t,x,a)\in[0,T]\times\R^n\times\R^{n+1}$ and $i\in\N$. We consider the closed balls
\begin{equation*}\begin{split}
&G_i(t,x,a):=\B(M_i(t,x)\,a,2d(M_i(t,x)\,a,E_{L_i}(t,x))),\\
&G(t,x,a):=\B(M(t,x)\,a,2d(M(t,x)\,a,E_{L}(t,x))).
\end{split}\end{equation*}

\noindent We notice that $\|G_i(t,x,a)\|\leq \varphi_i(t,x,a)$ and $\|G(t,x,a)\|\leq \varphi(t,x,a)$ for all $t\in[0,T]$, $x\in\R^n$, $a\in\R^{n+1}$, $i\in\N$, where
\begin{equation*}\begin{split}
&\varphi_i(t,x,a):= M_i(t,x)\,|a|+2d(M_i(t,x)\,a,E_{L_i}(t,x)),\\
&\varphi(t,x,a):= M(t,x)\,|a|+2d(M(t,x)\,a,E_{L}(t,x)).
\end{split}\end{equation*}

Let $P$ be the map defined in Lemma \ref{lem-pmh}. We define the following sets
\begin{equation*}\begin{split}
&\Phi_i(t,x,a):= P(M_i(t,x)\,a,E_{L_i}(t,x))=E_{L_i}(t,x)\cap G_i(t,x,a),\\
&\Phi(t,x,a):= P(M(t,x)\,a,E_{L}(t,x))=E_{L}(t,x)\cap G(t,x,a).
\end{split}\end{equation*}
By hypotheses and Corollary \ref{wrow-wm}, we get that the sets $\Phi_i(t,x,a)$, $\Phi(t,x,a)$ are
nonempty, compact, convex.

We define the single-valued maps \,$e_i,e$\, from \,$[0,T]\times\R^n\times\R^{n+1}$\, to \,$\R^{n+1}$\, by
\begin{equation}\label{stbprof0}
    e_i(t,x,a):= s_{n+1}(\Phi_i(t,x,a)),\quad e(t,x,a):= s_{n+1}(\Phi(t,x,a)),
\end{equation}
where $s_{n+1}$ in the Steiner selection. By Lemma \ref{lem-scmh} we have
\begin{equation}\label{stbprof1}
|e_i(t,x,a)-e(s,y,b)|\leq (n+1)\mathscr{H}(\Phi_i(t,x,a),\Phi(s,y,b))
\end{equation}
for all $t,s\in[0,T]$, $x,y\in\R^n$, $a,b\in\R^{n+1}$, $i\in\N$.

We notice that by the inequality \eqref{stbprof1} and Lemma \ref{lem-pmh}, we have
\begin{equation}\label{stbprof2}
|e_i(t,x,a)-e(t,x,a)|\leq 5(n+1)[\,\mathscr{H}(E_{L_i}(t,x),E(t,x))+|M_i(t,x)-M(t,x)|\,|a|\,].
\end{equation}
for all $t\in[0,T]$, $x\in\R^n$, $a\in\R^{n+1}$, $i\in\N$.

\begin{Th}\label{thm-zjnzz}
Let $H_i$, $H$, $L_i$, $L$, $M_i$, $M$,  $i\in\N$ be as above. If $H_i$ converge to $H$ and $M_i$ converge to $M$ uniformly on compacts, then
\begin{equation*}
e_i(t_i,x_i,a_i)\to e(t,x,a)\;\;\;\it{for every sequence}\;\;\; (t_i,x_i,a_i)\rightarrow (t,x,a).
\end{equation*}
\end{Th}

\begin{proof}
We notice that if
\begin{equation}\label{thm-zjnzz1}
\mathscr{H}(\Phi_i(t_i,x_i,a_i),\Phi(t,x,a))\to 0,\quad \forall\;(t_i,x_i,a_i)\rightarrow (t,x,a),
\end{equation}
then Theorem follows from the inequality \eqref{stbprof1}. Let $(t_i,x_i,a_i)\rightarrow (t,x,a)$. Then, using hypotheses, we have $M_i(t_i,x_i)\to M(t,x)$. Therefore by Proposition \ref{klconv} and the implication \eqref{zmkno} we obtain $\varphi_i(t_i,x_i,a_i)\to \varphi(t,x,a)$. Therefore there exists a constant $C>\varphi(t,x,a)$ such that  $\varphi_i(t_i,x_i,a_i)\leq C$ for every $i\in\N$. We have $\|G(t,x,a)\|\leq \varphi(t,x,a)$ and $\|G_i(t_i,x_i,a_i)\|\leq \varphi_i(t_i,x_i,a_i)$ for all $i\in\N$, so $\Phi(t,x,a)\subset G(t,x,a)\subset\B_C$ and $\Phi_i(t_i,x_i,a_i)\subset G_i(t_i,x_i,a_i)\subset\B_C$ for all $i\in\N$. It means that by  \eqref{zmkh}, the condition \eqref{thm-zjnzz1} is equivalent to the condition
\begin{equation}\label{thm-zjnzz2}
\lim_{i\to\infty}\Phi_i(t_i,x_i,a_i)=\Phi(t,x,a),\quad \forall\;(t_i,x_i,a_i)\rightarrow (t,x,a).
\end{equation}

Thus, to prove the theorem it is enough to show the condition \eqref{thm-zjnzz2}. Let $(t_i,x_i,a_i)\rightarrow (t,x,a)$. Then by Proposition \ref{klconv}, we obtain
\begin{equation}\label{thm-zjnzz3}
\lim_{i\to\infty}E_{L_i}(t_i,x_i)=E_L(t,x).
\end{equation}
By the inequality \eqref{dow-omk} and adequate convergence, we have that
\begin{equation*}
\mathscr{H}(G_i(t_i,x_i,a_i),G(t,x,a))\leq |\varphi_i(t_i,x_i,a_i)-\varphi(t,x,a)|+2|M_i(t_i,x_i)\,a_i- M(t,x)\,a|\to 0.
\end{equation*}
Besides, we know that $G(t,x,a)\subset\B_C$ and $G_i(t_i,x_i,a_i)\subset\B_C$ for all $i\in\N$ and some constant $C>0$. Then by \eqref{zmkh}, we have the equality
\begin{equation}\label{thm-zjnzz4}
    \lim_{i\to\infty}G_i(t_i,x_i,a_i)=G(t,x,a).
\end{equation}

If $\I G(t,x,a)\not=\emptyset$ then $E_L(t,x)\cap\I G(t,x,a)\not=\emptyset$. Therefore by the equality \eqref{thm-zjnzz3}, \eqref{thm-zjnzz4} and Theorem~\ref{thmozkp}, we have that $\lim_{i\to\infty}\Phi_i(t_i,x_i,a_i)=\Phi(t,x,a)$.

If $\I G(t,x,a)=\emptyset$ then $G(t,x,a)$ and $\Phi(t,x,a)$ are singletons
that contain points $M(t,x)\,a\in E_{L}(t,x)$. Let $y_i\in \Phi_i(t_i,x_i,a_i)$. Then $y_i\in G_i(t_i,x_i,a_i)$. Therefore by definition of $G_i(t_i,x_i,a_i)$ we have $|y_i-M_i(t_i,x_i)\,a_i|\leq 2d(M_i(t_i,x_i)\,a_i,E_{L_i}(t_i,x_i))$. Obviously, $M_i(t_i,x_i)\,a_i\to M(t,x)\,a$. Therefore by the equality \eqref{thm-zjnzz3} and the implication \eqref{zmkno} we obtain $2d(M_i(t_i,x_i)\,a_i,E_{L_i}(t_i,x_i))\to 2d(M(t,x)\,a,E_{L}(t,x))=0$. Therefore $y_i\to M(t,x)\,a$. It means $M(t,x)\,a\in\liminf_{i\to\infty}\Phi_i(t_i,x_i,a_i)$. So by the equality
\eqref{thm-zjnzz4}
\begin{equation*}
\{M(t,x)\,a\}\subset\liminf_{i\to\infty}\Phi_i(t_i,x_i,a_i)
    \subset\limsup_{i\to\infty}\Phi_i(t_i,x_i,a_i)\subset G(t,x,a)=\{M(t,x)\,a\}.
\end{equation*}
Therefore $\lim_{i\to\infty}\Phi_i(t_i,x_i,a_i)=\{M(t,x)\,a\}=\Phi(t,x,a)$ that ends the proof.
\end{proof}

\begin{Th}\label{thm-zjnzz26}
Let $H_i$, $H$, $L_i$, $L$,   $i\in\N$ be as above.  If  $H_i$ and  $H$ satisfy the condition \eqref{thm-rep-stab5-11} then
\begin{equation*}
\sup_{(t,x,a)\,\in\,[0,T]\times\B_R\times\R^{n+1}}|e_i(t,x,a)-e(t,x,a)|
\;\xrightarrow[i\to\infty]{}\;0,\quad\forall\;R>0.
\end{equation*}
\end{Th}

\begin{proof}
Let $M_i(t,x)=M(t,x)\equiv 1$. B the inequality \eqref{stbprof2} and Proposition \ref{gklconv}, we obtain
\begin{align*}
\sup_{(t,x,a)\,\in\,[0,T]\times\B_R\times\R^{n+1}}|e_i(t,x,a)-e(t,x,a)| \;\;&\leq\;\;  5(n+1)\sup_{(t,x)\,\in\,[0,T]\times\B_R}\mathscr{H}(E_{L_i}(t,x),E(t,x))\\ &\xrightarrow[\;i\to\infty\;]{}\;0
\end{align*}
 for every $R>0$, that ends the proof.
\end{proof}

\begin{Rem}
Let $e_i,e:[0,T]\times\R^n\times\R^{n+1}\rightarrow\R^{n+1}$ be functions given by formulas
 \eqref{stbprof0}. For all $i\in\N$ we define the functions $f_i,f:[0,T]\times\R^n\times \R^{n+1}\rightarrow\R^n$ and $l_i,l:[0,T]\times\R^n\times \R^{n+1}\rightarrow\R$ by
formulas:
\begin{equation*}\begin{split}
&f_i(t,x,a):=\pi_v(e(t,x,a))\;\;\;\te{and}\;\;\; l_i(t,x,a):=\pi_\eta(e(t,x,a)),\\
&f(t,x,a):=\pi_v(e(t,x,a))\;\;\;\te{and}\;\;\; l(t,x,a):=\pi_\eta(e(t,x,a)),
\end{split}\end{equation*}
where $\pi_v(v,\eta)=v$ and $\pi_\eta(v,\eta)=\eta$ for every $v\in\R^n$ and $\eta\in\R$.
Then for every $t\in[0,T]$, $x\in\R^n$, $a\in \R^{n+1}$, $i\in\N$ the following
equalities hold:
\begin{equation}\label{dfrz}
e_i(t,x,a)=(f_i(t,x,a),l_i(t,x,a)), \quad e(t,x,a)=(f(t,x,a),l(t,x,a)).
\end{equation}
The equality \eqref{dfrz} and Theorem \ref{thm-zjnzz} imply Theorems \ref{thm-rep-stab1} and
 \ref{thm-rep-stab2} if instead of $M_i(t,x)$, $M(t,x)$ we take $M_i(t,x)=M(t,x)\equiv 1$ and
\begin{equation*}
\begin{split}
    &M_i(t,x):=|\lambda_i(t,x)|+|H_i(t,x,0)|+c_i(t)(1+|x|)+1\\
    &M(t,x):=|\lambda(t,x)|+|H(t,x,0)|+c(t)(1+|x|)+1,
\end{split}
\end{equation*}
respectively. By the equality \eqref{dfrz} and Theorem  \ref{thm-zjnzz26}, we get
Theorem  \ref{thm-rep-stab5} if instead of $M_i(t,x)$, $M(t,x)$ we take   $M_i(t,x)=M(t,x)\equiv 1$.

Theorems \ref{thm-rep-stab3}, \ref{thm-rep-stab4}, \ref{thm-rep-stab6} can be proven similarly as above, indeed, it is enough to fix $t\in[0,T]$.
\end{Rem}






\end{document}